\documentclass[12pt]{amsart}
\usepackage{amssymb, eucal, amsfonts, amsmath, xypic,latexsym, amscd}

\def\k{{\Bbbk}}

\def\g{{\mathfrak g}}

\def\h{{\mathfrak h}}

\def\t{{\mathfrak t}}

\def\Q{{\mathcal Q}}

\def\m{{\mathfrak m}}

\def\n{{\mathfrak n}}
\def\p{{\mathfrak p}}

\def\K{{\mathbb K}_p}
\def\Z{{\mathbb Z}}

\def\Lie{\mathop{\fam0 Lie}}

\newcommand{\too}{\,\,\longrightarrow\,\,}

\theoremstyle{plain}
\newtheorem{theorem}{Theorem}[section]
\newtheorem{corollary}{Corollary}[section]
\newtheorem{prop}{Proposition}[section]

\theoremstyle{definition}

\theoremstyle{remark}

\newtheorem{rem}{Remark}[section]

\def\subtitle#1. {{\medskip\bf#1\par\nobreak\smallskip}}
\def\proclaim#1. {\medbreak\bgroup\noindent\bf#1. \it}

\def\endproclaim{\egroup
\ifdim\lastskip<\medskipamount\removelastskip\medskip\fi}
\newcount
=0
\def\citedef#1 {\advance\citation by1
  \expandafter\edef\csname#1\endcsname{{\the\citation}}
  \checkendcitedef}
\def\checkendcitedef#1{\ifx#1\endcitedef\else\citedef#1\fi}
\def\cite#1{\csname#1\endcsname}
\citedef  AOV AOV1 Bois Bo BoSp BoT BGR BH C1 C2 CT1 CTK CT CS Co DG
Dix Eis GK1 GK2 Jan97 Jo Ha Hu Hu1 KW Kry Mc MR Ng N O PT Po Sh2 St
T Ti Ve Vo1 Vo2 Zass
\endcitedef
\newtoks\nextauth
\newif\iffirstauth
\def\checkendauth#1{\ifx\endauth#1
        \iffirstauth\the\nextauth
        \else{} and \the\nextauth\fi,
    \else\iffirstauth\the\nextauth\firstauthfalse
        \else, \the\nextauth\fi
        \expandafter\auth\expandafter#1\fi}
\def\auth#1 #2 {\nextauth={#1 #2}\checkendauth}
\newif\ifinbook
\newif\ifbookref
\def\nextref#1 {\bookreffalse\inbookfalse
    \bibitem[\cite{#1}]{}
    \firstauthtrue
    \ignorespaces}
\def\paper#1{{\it#1,}}
\def\In#1{\inbooktrue In #1,}
\def\book#1{\bookreftrue{\it#1,}}
\def\journal#1{#1\ifinbook,\fi}
\def\bookseries#1{#1,}
\def\Vol#1{\ifbookref Vol. #1,\else\ifinbook Vol. #1,\else{\bf#1}\fi\fi
    \space\ignorespaces}
\def\nombre#1{no. #1}
\def\publisher#1{#1,}
\def\Year#1{\ifbookref #1.\else\ifinbook #1,\else(#1)\fi\fi
    \space\ignorespaces}
\def\Pages#1{\ifinbook pp. #1.\else #1.\fi}
\begin{document}
\title{Modular Lie algebras and the Gelfand--Kirillov conjecture}
\author{Alexander Premet}
\thanks{\nonumber{\it Mathematics Subject Classification} (2000 {\it revision}).
Primary 17B35. Secondary 17B20, 17B50.}
\address{School of Mathematics, University of Manchester, Oxford Road,
M13 9PL, UK} \email{sashap@maths.man.ac.uk}
\begin{abstract}
\noindent Let $\g$ be a finite dimensional simple Lie algebra over
and algebraically closed field $\mathbb K$ of characteristic $0$.
Let $\g_{\Z}$ be a Chevalley $\Z$-form of $\g$ and
$\g_{\k}=\g_{\Z}\otimes_{\Z}\,\k$, where $\k$ is the algebraic
closure of ${\mathbb F}_p$. Let $G_\k$ be a simple, simply connected
algebraic $\k$-group with $\Lie(G_\k)\,=\,\g_\k$. In this paper, we
apply recent results of Rudolf Tange on the fraction field of the
centre of the universal enveloping algebra $U(\g_\k)$ to show that
if the Gelfand--Kirillov conjecture (from 1966) holds for $\g$, then
for all $p\gg 0$ the function field $\k(\g_\k)$ on the dual space
$\g_\k$ is purely transcendental over its subfield
$\k(\g_\k)^{G_\k}$. Very recently, it was proved by
Colliot-Th{\'e}l{\`e}ne--Kunyavski$\breve{\i}$--Popov--Reichstein
that the function field ${\mathbb K}(\g)$ is {\it not} purely
transcendental over its subfield ${\mathbb K}(\g)^\g$ provided that
$\g$ is of type ${\rm B}_n$, $n\ge 3$,\, ${\rm D}_n$, $n\ge 4$,\,
${\rm E}_6$,\, ${\rm E}_7$,\, ${\rm E}_8$ or ${\rm F}_4$. We prove a
modular version of this result (valid for $p\gg 0$) and use it to
show that, in characteristic $0$, the Gelfand--Kirilov conjecture
fails for the simple Lie algebras of the above types. In other
words, if $\g$ of type ${\rm B}_n$, $n\ge 3$,\, ${\rm D}_n$, $n\ge
4$,\, ${\rm E}_6$,\, ${\rm E}_7$,\, ${\rm E}_8$ or ${\rm F}_4$, then
the Lie field of $\g$ is more complicated than expected.
\end{abstract}
\maketitle

\section{\bf Introduction and preliminaries}\label{intro}

\smallskip

\subsection{}
\noindent Let $\mathbb{K}$ be an algebraically closed field. Given a
Lie algebra $L$ over $\mathbb{K}$ we denote by $U(L)$ the universal
enveloping algebra of $L$. Since $U(L)$ is a Noetherian domain, it
admits a field of fraction which we shall denote by
$\mathcal{D}(L)$. Let ${\bf A}_r(\mathbb{K})$ denote the $r$-th Weyl
algebra over $\mathbb K$ (it is generated over $\mathbb K$ by $2r$
generators $u_1,\ldots, u_r, v_1,\ldots, v_r$ subject to the
relations $[u_i,u_j]=[v_i,v_j]=0$ and $[u_i,v_j]=\delta_{ij}$ for
all $i,j\le r$). Given a collection of free variables $y_1,\ldots,
y_s$ we define $${\mathbf
A}_{r,s}(\mathbb{K}):=\mathbf{A}_r(\mathbb{K})\otimes\mathbb{K}[y_1,\ldots,y_s].$$
Being a Noetherian domain the algebra $\mathbf{A}_{r,s}(\mathbb{K})$
also admits a field of fractions denoted
$\mathcal{D}_{r,s}(\mathbb{K})$.

\medskip

In [\cite{GK1}], Gelfand and Kirillov put forward the following {\it
Hypoth{\`e}se fondamentale}:

\medskip

\noindent {\sc The Gelfand--Kirillov conjecture.} {\it If ${\rm
char}(\mathbb{K})=0$ and $L$ is the Lie algebra of an algebraic
$\mathbb{K}$-group, then $\mathcal{D}(L)\,\cong
\,\mathcal{D}_{r,s}(\mathbb{K})$ for some $r,s$ depending on $L$}.

\medskip

\noindent If the Gelfand--Kirillov conjecture holds for $L$, then
necessarily $$s={\rm index}\,L={\rm
tr.\,deg}(Z(\mathcal{D}(L)),\quad r=\frac{1}{2}(\dim\,L-{\rm
index}\,L),
$$ where $Z(\mathcal{D}(L))$ is the centre of $\mathcal{D}(L)$; see [\cite{O}] for more
detail.

In [\cite{GK1}], the conjecture was settled for nilpotent Lie
algebras, $\mathfrak{sl}_n$ and $\mathfrak{gl}_n$. In 1973, the
conjecture was confirmed in the solvable case independently by Borho
[\cite{BGR}], Joseph [\cite{Jo}] and McConnell [\cite{Mc}]. In 1979,
Nghiem considered the semi-direct products of $\mathfrak{sl}_n$,
$\mathfrak{sp}_{2n}$ and $\mathfrak{so}_{n}$ with their standard
modules and proved the conjecture for those; see [\cite{Ng}].

\smallskip

\subsection{}
\noindent A breakthrough in the general case came in 1996 when
Jacques Alev, Alfons Ooms and Michel Van den Bergh constructed a
series of counterexamples to the conjecture, focusing on semi-direct
products of the form $L=\Lie(H)\ltimes V$ where $H$ is a simple
algebraic group and $V$ is a rational $H$-module admitting a trivial
generic stabilizer ($V$ is regarded as an abelian ideal of $L$). The
smallest known counterexample is the $9$-dimensional semi-direct
product of $\mathfrak{sl}_2$ with a direct sum of two copies of the
adjoint module. In [\cite{AOV1}], Alev, Ooms and Van den Bergh
proved that the conjecture holds in dimension $\le 8$.

However, despite considerable efforts the validity of the
Gelfand--Kirillov conjecture in the case of a {\it simple} Lie
algebra $L\not\cong\mathfrak{sl}_n$ remained a complete mystery
until now. It suffices to say that the answer is unknown already for
$L=\mathfrak{sp}_4$. A weaker positive result in the case of $L$
simple was obtained by Gelfand and Kirillov in 1968. They proved in
[\cite{GK2}] that there exists a finite field extension $F$ of the
centre $Z(\mathcal{D}(L))$ such that the field of fractions of
$\mathcal{D}(L)\otimes_{Z(\mathcal{D}(L))}\,F$ is isomorphic to
$\mathcal{D}_{N,l}(\mathbb{K})$, where $l$ is the rank of $L$ and
$N=\frac{1}{2}(\dim L-l)$. It is conjectured in [\cite{AOV}] that
such a weakened version of the conjecture should hold for any
algebraic Lie algebra $L$. At the opposite extreme, it was proved in
[\cite{Co}] for $L$ simple that the obvious analogue of the
Gelfand--Kirillov conjecture holds for the fraction fields of the
largest primitive quotients of $U(L)$.

\smallskip

\subsection{}
\noindent As the author first learned from Jacques Alev,  the
Gelfand--Kirillov conjecture makes perfect sense in the case where
the base field $\mathbb K$ has characteristic $p>0$ (there is no
need to make any changes in the formulation as all objects involved
exist in any characteristic). In principle, the problem can be
stated for any finite dimensional restricted Lie algebra, but in
what follows I am going to focus on the case where $L=\g_p$ is the
Lie algebra of a simple, simply connected algebraic
$\mathbb{K}$-group $G_p$.

The Lie algebra $\g_p=\Lie(G_p)$ carries a canonical $p$-th power
map $x\mapsto x^{[p]}$ equivariant under the adjoint action of
$G_p$. The elements $x^p-x^{[p]}$ with $x\in \g$ generate a large
subalgebra of the centre $Z(\g_p)$ of the universal enveloping
algebra $U(\g)$, called the $p$-centre of $U(\g_p)$ and denoted
$Z_p(\g)$. It follows from the PBW theorem that $U(\g_p)$ is free
module of finite rank over $Z_p(\g_p)$. Let $\Q(\g_p)$ denote the
field of fractions of $Z(\g_p)$. It is well known that under very
mild assumptions on $G_p$ one has that $\mathcal{D}(\g_p)\,\cong\,
U(\g_p)\otimes_{Z(\g_p)}\Q(\g_p)$ is a central division algebra of
dimension $p^{n-l}$ over the field $\Q(\g_p)$, where $n=\dim\,\g_p$
\and $l={\rm rk}\,G_p$; see [\cite{Zass}, \cite{KW}] for more
detail.

It is known (and easily seen) that if the Gelfand--Kirillov
conjecture hold for $\g_p$, then the field $\Q(\g_p)$ is purely
transcendental over $\mathbb K$ and the order of the similarity
class of $\mathcal{D}(\g_p)$ in the Brauer group ${\rm
Br}(\Q(\g_p))$ equals $p$; see [\cite{PT}, \cite{Bois}] for more
detail. At the Durham Symposium on Quantum Groups in July 1999, Alev
asked the author whether the field $\Q(\g_p)$ is purely
transcendental over $\mathbb K$. The question was, no doubt,
motivated by the Gelfand--Kirillov conjecture.

In [\cite{PT}], Rudolf Tange and the author answered Alev's question
in affirmative for $\g_p=\mathfrak{gl}_n$ and for
$\g_p=\mathfrak{sl}_n$ with $p\nmid n$. Using our result Jean-Marie
Bois was able to confirm the modular Gelfand--Kirillov conjecture in
these cases; see [\cite{Bois}]. Recently, Tange [\cite{T}] solved
Alev's problem for any simple, simply connected group $G_p$ subject
to some (very mild) assumptions on $p$. In [\cite{T}], he also
proved that the centre $Z(\g_p)$ is a unique factorisation domain,
thus confirming an earlier conjecture of Braun--Hajarnavis; see
[\cite{BH}, Conjecture~E].

\smallskip

\subsection{}
\noindent Let $\g$ be a characteristic $0$ counterpart of $\g_p$, a
simple Lie algebra which has the same root system as $G_p$. Although
proving the Gelfand--Kirillov conjecture for $\g_p$ would probably
have little impact on its validity for $\g$ (apart from some
heuristic evidence), it turns out that {\it disproving} the
conjecture for $\g_p$ for almost all $p$ is sufficient for refuting
the original conjecture for $\g$.

In what follows we assume that $\mathbb K$ is an algebraically
closed field of characteristic $0$ and denote by $\k$ the algebraic
closure of the prime field ${\mathbb F}_p$. We let $\g_\Z$ be a
Chevalley $\Z$-form associated with a minimal admissible lattice in
$\g$ and set $\g_\k:=\g_\Z\otimes_\Z\k$. Then
$\g\cong\g_\Z\otimes_\Z\mathbb{K}$ and $\g_\k=\Lie(G_\k)$ for some
simple, simply connected algebraic $\k$-group $G_\k$ of the same
type as $\g$.

In Section~2 we prove a reduction theorem which states that if the
Gelfand--Kirillov conjecture holds for $\g$, then it holds for
$\g_\k$ for almost all $p$. In Section~3, we apply Tange's results
[\cite{T}] to show that if the modular Gelfand--Kirillov conjecture
holds for $\g_\k$, then the field $\k(\g_\k)$ of rational functions
on $\g_\k$ is purely transcendental over the field of invariants
$\k(\g_\k)^{G_\k}$.

Incidently, it was recently proved by Jean-Lois
Colliot-Th{\'e}l{\`e}ne, Boris Kunyavski$\breve{\i}$, Vladimir Popov
and Zinovy Reichstein that if the function field ${\mathbb K}(\g)$
is purely transcendental over the field of invariants ${\mathbb
K}(\g)^\g$, then $\g$ is of type ${\rm A}_n$, ${\rm C}_n$ or ${\rm
G}_2$; see [\cite{CT}, Thm.~0.2(b)]. In Section~4, we prove a
modular version of this result valid for $p\gg 0$. As a consequence,
we obtain the following:
\begin{theorem}
Let $\g$ be a finite dimensional simple Lie algebra over an
algebraically closed field $\mathbb K$ of characteristic $0$. If
$\mathcal{D}(\g)\,\cong\, \mathcal{D}_{r,s}(\mathbb{K})$ for some
$r,s$, then $\g$ is of type ${\rm A}_n$, ${\rm C}_n$ or ${\rm G}_2$.
\end{theorem}
This shows that the original Gelfand--Kirillov conjecture does not
hold for simple Lie algebras of types ${\rm B}_n$, $n\ge 3$, ${\rm
D}_n$, $n\ge 4$, ${\rm E}_6$, ${\rm E}_7$, ${\rm E}_8$ and ${\rm
F}_4$. It seems plausible to the author that the conjecture {\it
does} hold for simple Lie algebras of type $\rm C$. The supporting
evidence for that comes from [\cite{CT}, Thm.~0.2(a)] which says
that in type $\rm C$ the field $\mathbb{K}(\g)$ is purely
transcendental over its subfield ${\mathbb K}(\g)^{\g}$. Some of the
results obtained in [\cite{Ng}] might be useful for proving the
conjecture in type $\rm C$.

\medskip

\noindent{\bf Acknowledgement.} Part of this work was done during my
stay at the Isaac Newton Institute (Cambridge) in May--June 2009. I
would like to thank the Institute for warm hospitality and excellent
working conditions. The results of this paper were announced in my
talk at the final conference of the EPSRC Programme ``Algebraic Lie
Theory'' held at the INI in June 2009. I would like to thank Boris
Kunyavski$\breve{\i}$, Vladimir Popov, Andrei Rapinchuk and Rudolf
Tange for some very helpful email correspondence.
\section{\bf The Gelfand--Kirillov conjecture and its modular analogues}\label{sec1}
\subsection{} In this paper we treat the Gelfand--Kirillov conjecture as
a noncommutative version of a purity problem for field extensions.
In order to reduce it to a classical purity problem, as studied in
birational invariant theory, we seek a passage to finite
characteristics. As a first step, we make a transit from $\g$ to
$\g_\k$ ensuring in advance that the validity of the
Gelfand-Kirillov conjecture for $\g$ implies that for $\g_\k$. Since
$U(\g_\k)$ is a finite module over its center $Z(\g_\k)$, the field
of fractions $\mathcal{D}(\g_\k)$ is a finite dimensional central
division algebra over the fraction field $\Q(\g_\k)$ of $Z(\g_\k)$.
This enables us to  apply recent results of Tange [\cite{T}] on the
rationality of $\Q(\g_\k)$ to reduce the original problem about the
structure of $\mathcal{D}(\g)$ to the purity problem for the field
extension $\k(\g_\k)/\k(\g_\k)^{G_\k}$.

\subsection{} In this subsection we prove our reduction theorem:
\begin{theorem}\label{GK mod p}
If the Gelfand--Kirillov conjecture holds for $\g$, then it holds
for $\g_\k$ for all $p\gg 0$, where $\k$ is the algebraic closure of
${\mathbb F}_p$.
\end{theorem}
\begin{proof}
(A) Choose a Chevalley basis ${\mathcal B}=\{x_1,\ldots, x_n\}$ of
$\g_\Z$ and denote by $U_d(\g)$ the $d$-th component of the
canonical filtration of $U(\g)$. If the field of fractions
$\mathcal{D}(\g)$ is isomorphic to {\sl Frac}\,$\big({\bf
A}_N\otimes Z(\g)\big)$, where $N$ is the number of positive roots
of $\g$, then there exist $w_1,\ldots, w_{2N}\in \mathcal{D}(\g)$
such that
\begin{eqnarray}
[w_i,w_j]\,=\,[w_{N+i},w_{N+j}]&=&0\,\,\, \qquad \quad(1\le i,j\le N);\label{1}\\
\,[w_i,w_{N+j}]&=&\delta_{i,j}\,\, \qquad \,\, (1\le i,j\le N);\label{2}\\
Q_k\cdot x_k&=&P_k,\qquad \, \,\,(1\le k\le n)\label{3}
\end{eqnarray}
for some {\it nonzero} polynomials $P_i,Q_i$ in $w_1,\ldots, w_{2N}$
with coefficients in $Z(\g)$ (here (\ref{3}) follows from the fact
that the monomials $w_1^{a_1}w_2^{a_2}\cdots w_{2N}^{a_{2N}}$ with
$a_i\in\Z_+$ form a basis of the $\k$-subalgebra of ${\mathcal
D}(\g)$ generated by $w_1,\ldots,w_{2N}$). Since $w_i=v_i^{-1}u_i$
for some {\it nonzero} elements $u_i,v_i\in U_{d(i)}(\g)$, we can
rewrite (\ref{1}) and (\ref{2}) as follows
\begin{eqnarray}
v_i^{-1}u_i\cdot v_j^{-1}u_j&=&v_j^{-1}u_j\cdot v_i^{-1}u_i;\label{1'}\\
v_{N+i}^{-1}u_{N+i}\cdot v_{N+j}^{-1}u_{N+j}&=&v_{N+j}^{-1}u_{N+j}\cdot v_{N+i}^{-1}u_{N+i};\label{2'}\\
v_{i}^{-1}u_{i}\cdot v_{N+j}^{-1}u_{N+j}-v_{N+j}^{-1}u_{N+j}\cdot v_{i}^{-1}u_{i}&=
&\delta_{i,j}\qquad\qquad \quad\quad \ (1\le i,j\le N). \label{3'}
\end{eqnarray}
As the nonzero elements of $U(\g)$ form an Ore set, there are {\it
nonzero} elements $v_{i,j}, u_{i,j}\in U_{d(i,j)}(\g)$ such that
\begin{equation}\label{uv}
v_{i,j}u_i=u_{i,j}v_j\qquad\qquad\quad(1\le i,j\le 2N).
\end{equation}
Thus we can rewrite (\ref{1'}), (\ref{2'}) and (\ref{3'})
in the form
\begin{eqnarray}
\quad v_i^{-1}v_{i,j}^{-1}\cdot
u_{i,j}u_j&=&v_j^{-1}v_{j,i}^{-1}\cdot u_{j,i}u_i
\qquad(1\le i,j\le N\mbox{ or } N\le i,j\le 2N)\label{1''}\\
\quad v_{i}^{-1}v_{i,N+j}^{-1}\cdot u_{i,N+j}u_{N+j}&=&\delta_{ij}+v_{N+j}^{-1}v_{N+j,i}^{-1}\cdot
u_{N+j,i}u_{N+i}\qquad(1\le i,j\le N).\label{2''}
\end{eqnarray}
By the same reasoning, there exist {\it nonzero} elements $a_{i,j},
b_{i,j}\in U_{d(i,j)}(\g)$ such that
\begin{equation}\label{10}
a_{i,j}v_{i,j}v_i=b_{i,j}v_{j,i}v_j\qquad\qquad\quad(1\le i,j\le
2N).
\end{equation}
Since $$v_{i,j}v_i(v_{j,i}v_j)^{-1}\,=\,a_{i,j}^{-1}b_{i,j},$$ it is
straightforward to see that (\ref{1''}) and (\ref{2''}) can be
rewritten as
\begin{eqnarray}
\quad a_{i,j}u_{i,j}u_j&=&b_{i,j}u_{j,i}u_i
\qquad(1\le i,j\le N\mbox{ or } N\le i,j\le 2N)\label{1'''}\\
\qquad a_{i,N+j}u_{i,N+j}u_{N+j}&=&\delta_{ij}a_{i,N+j}v_{i,N+j}v_i+
b_{i,N+j} u_{N+j,i}u_{i}\qquad(1\le i,j\le N).\label{2'''}
\end{eqnarray}

For an $m$-tuple ${\bf i}=(i(1),i(2),\ldots,i(m))$ with $1\le
i(1)\le i(2)\le \cdots\le i(m)\le 2N$ and $m\ge 3$ we select
(recursively) {\it nonzero} elements $u_{i(1),\ldots, i(k)},
v_{i(1),\ldots,i(k)}\in U_{d(\bf i)}(\g)$, where $3\le k\le m$, such
that
\begin{equation}\label{13}
v_{i(1),\ldots, i(k)}u_{i(1),\ldots,i(k-1)}u_{i(k-1)}=u_{i(1),\ldots, i(k)}v_{i(k)}.
\end{equation}
Write $w^{\mathbf i}:=w_{i(1)}\cdot w_{i(2)}\cdot\ldots\cdot w_{i(m)}
=\prod_{k=1}^m\,v_{i(k)}^{-1}u_{i(k)}.$ Then
\begin{eqnarray*}
w^{\mathbf i}&=&v_{i(1)}^{-1}u_{i(1)}\cdot
v_{i(2)}^{-1}u_{i(2)}\cdot{\prod}_{k=3}^m\,v_{i(k)}^{-1}u_{i(k)}\\
&=&v_{i(1)}^{-1}v_{i(1),i(2)}^{-1}u_{i(1),i(2)}u_{i(2)}\cdot v_{i(3)}^{-1}u_{i(3)}\cdot
\textstyle{\prod}_{k=4}^m\,v_{i(k)}^{-1}u_{i(k)}\\
&=&v_{i(1)}^{-1}v_{i(1),i(2)}^{-1}v_{i(1),i(2),i(3)}^{-1}
u_{i(1),i(2),i(3)}u_{i(3)}\cdot
\textstyle{\prod}_{k=4}^m\,v_{i(k)}^{-1}u_{i(k)}\\
&=&\cdots\,=\,\Big({\prod}_{k=1}^m\,v_{i(1),\ldots, i(m-k+1)}
\Big)^{-1}\cdot
u_{i(1),\ldots, i(m)}u_{i(m)}.
\end{eqnarray*}
We now put $v_{\bf i}:={\prod}_{k=1}^m\,v_{i(1),\ldots, i(m-k+1)}$ and $u_{\bf i}:=u_{i(1),\ldots, i(m)}u_{i(m)}$.

Let $\{{\bf i}(1),\ldots,{\bf i}(r)\}$ be the set of all tuples as above with
$\sum_{\ell=1}^mi(\ell)\le M$,
where $M=\max\{\deg P_i,\,\deg Q_i\,|\,\,1\le i\le n\}$. Clearly,
$P_k=\sum_{j=1}^r\,\lambda_{j,k}w^{{\bf i}(j)}$ and
$Q_k=\sum_{j=1}^r\,\mu_{j,k}w^{{\bf i}(j)}$ for some
$\lambda_{j,k},\mu_{j,k}\in Z(\g)$, where $1\le k\le n$.
The above discussion then shows that $P_k=\sum_{j=1}^r\,\lambda_{j,k}\,v_{{\bf i}(j)}^{-1}
u_{{\bf i}(j)}$ and $Q_k=\sum_{i=1}^r\,
\mu_{j,k}\,v_{{\bf i}(j)}^{-1}u_{{\bf i}(j)}$.

It is well known that $Z(\g)$, the centre of $U(\g)$, is freely generated over $\mathbb K$ by
$l={\rm rk}\, \g$
elements $\psi_1,\ldots, \psi_l\in
U(\g_\Z)$. Moreover, for
$p\gg 0$ the invariant algebra $U(\g_\k)^{G_\k}\subset Z(\g_\k)$ with respect to the adjoint action of
$G_\k$ is freely generated over $\k$ by $\overline{\psi}_1,\ldots, \overline{\psi}_l$, the images of
$\psi_1,\ldots,\psi_l$ in $U(\g_\k)=U(\g_\Z)\otimes_\Z\k$; see [\cite{Jan97}, 9.6] for instance.

We can write $$\lambda_{j,k}=\sum_{\bf a}\,\lambda_{j,k}(a_1,\ldots,a_l)\psi_1^{a_1}\cdots\psi_l^{a_l}
\ \mbox{ and }\
\mu_{j,k}=\sum_{\bf a}\mu_{j,k}(a_1,\ldots,a_l)\psi_1^{a_1}\cdots\psi_l^{a_l}$$ for some scalars
$\lambda_{j,k}(a_1,\ldots,a_l),\,\mu_{j,k}(a_1,\ldots,a_l)
\in\mathbb K$, where the summation runs over finitely many $l$-tuples ${\bf a}=(a_1,\ldots,a_l)
\in\Z_+^l$.

There exist {\it nonzero} $c_{{\bf i}(j);\,k},\,d_{{\bf
i}(j);\,k}\in U_{d({\bf i}(j),k)}(\g)$ such that
\begin{equation}\label{14}
u_{{\bf i}(j)}x_k d_{{\bf i}(j);\,k}\,=\,v_{{\bf i}(j)}c_{{\bf
i}(j);\,k}\qquad \quad (1\le j\le r,\,1\le k\le n).
\end{equation}

Since $P_k=Q_kx_k$, we have that
\begin{equation}\label{15}
\sum_{j=1}^r\,\lambda_{j,k}\,v_{{\bf i}(j)}^{-1} u_{{\bf
i}(j)}\,=\,\sum_{i=1}^r\, \mu_{j,k}\,c_{{\bf i}(j);\,k}d_{{\bf
i}(j);\,k}^{-1}\qquad(1\le k\le n).
\end{equation}
Set $v_{{\bf i}(j)}(0):=v_{{\bf i}(j)}$, $u_{{\bf i}(j)}(0)=u_{{\bf
i}(j)}$, $c_{{\bf i}(j);\,k}(0):=c_{{\bf i}(j);\,k}$, $d_{{\bf
i}(j);k}(0):=d_{{\bf i}(j);k}$. For each pair $(j,s)$ of positive
integers satisfying $r\ge j>s>0$ we select (recursively) {\it
nonzero} elements $v_{{\bf i}(j)}(s),\, u_{{\bf i}(j)}(s),\, c_{{\bf
i}(j);\,k}(s), \,d_{{\bf i}(j);\,k}(s)\in U_{d({\bf i}(j),k,s)}(\g)$
such that
\begin{eqnarray}
v_{{\bf i}(j)}(s)v_{{\bf i}(s)}(s-1)&=&u_{{\bf i}(j)}(s)v_{{\bf i}(j)}(s-1)\label{16}\\
d_{{\bf i}(j);\,k}(s-1)c_{{\bf i}(j);\,k}(s)&=&d_{{\bf
i}(s);\,k}(s-1)d_{{\bf i}(j);\,k}(s)\label{17}.
\end{eqnarray}
Multiplying both sides of (\ref{15}) by $v_{{\bf i}(1)}$ on the left
and by $d_{{\bf i}(1),k}$ on the right we obtain (after applying
(\ref{16}) and (\ref{17}) with $s=1$) that
\begin{eqnarray*}
0&=&\lambda_{1,k}u_{{\bf i}(1)}d_{{\bf i}(1);\,k}
-\mu_{1,k}v_{{\bf i}(1)}c_{{\bf i}(1);\,k}\\
&+&\sum_{j=2}^r\big(\lambda_{j,k}v_{{\bf i}(1)}v_{{\bf i}(j)}^{-1}
u_{{\bf i}(j)}d_{{\bf i}(1);\,k}-\mu_{j,k}v_{{\bf i}(1)}c_{{\bf
i}(j);\,k}
d_{{\bf i}(j);k}^{-1}d_{{\bf i}(1);\,k}\big)\\
&=&\lambda_{1,k}u_{{\bf i}(1)}d_{{\bf i}(1);\,k}
-\mu_{1,k}v_{{\bf i}(1)}c_{{\bf i}(1);\,k}\\
&+&\sum_{j=2}^r\big(\lambda_{j,k}v_{{\bf i}(j)}(1)^{-1}u_{{\bf
i}(j)}(1) u_{{\bf i}(j)}d_{{\bf i}(1);\,k}-\mu_{j,k}v_{{\bf
i}(1)}c_{{\bf i}(j);k} c_{{\bf i}(j);\,k}(1)d_{{\bf
i}(1);\,k}(1)^{-1}\big).
\end{eqnarray*}
Multiplying both sides of this equality by $v_{{\bf i}(2)}(1)$ on
the left and by $d_{{\bf i}(2);\,k}(1)$ on the right and applying
(\ref{16}) and (\ref{17}) with $s=2$ we get
\begin{eqnarray*}
0&=&\lambda_{1,k}v_{{\bf i}(2)}(1)u_{{\bf i}(1)}d_{{\bf
i}(1);\,k}d_{{\bf i}(2);k}(1)
-\mu_{1,k}v_{{\bf i}(2)}(1)v_{{\bf i}(1)}c_{{\bf i}(1);\,k}d_{{\bf i}(2);\,k}(1)\\
&+&\lambda_{2,k}u_{{\bf i}(2)}(1)u_{{\bf i}(1)}d_{{\bf
i}(1);\,k}d_{{\bf i}(2);\,k}(1)
-\mu_{2,k}v_{{\bf i}(2)}(1)v_{{\bf i}(1)}c_{{\bf i}(2);\,k}c_{{\bf i}(2);\,k}(1)\\
&+&\sum_{j=3}^r\lambda_{j,k}v_{{\bf i}(j)}(2)^{-1}u_{{\bf i}(j)}(2)u_{{\bf i}(j)}(1)
u_{{\bf i}(j)}d_{{\bf i}(1);\,k}d_{{\bf i}(2);\,k}(1)\\
&-&\sum_{j=3}^r\mu_{j,k}v_{{\bf i}(2)}(1)v_{{\bf i}(1)}c_{{\bf
i}(j);\,k} c_{{\bf i}(j);k}(1)c_{{\bf i}(j);\,k}(2)d_{{\bf
i}(2);\,k}(2)^{-1}.
\end{eqnarray*}

Repeating this process $r$ times we get rid of all denominators and
arrive at the equality
\begin{eqnarray}
&&\Big(\sum_{j=1}^r\lambda_{j,k}\prod_{\ell=1}^{r-j}v_{{\bf
i}(r-\ell+1)}(r-\ell)\prod_{\ell=1}^{j} u_{{\bf
j}(j-\ell+1)}(j-\ell)\Big)\prod_{\ell=1}^{r}d_{{\bf i}(\ell);k}(\ell-1)
\label{18}\,=\\
&=&\Big(\prod_{\ell=1}^{r} v_{{\bf
i}(r-\ell+1)}(r-\ell)\Big)\Big(\sum_{j=1}^r\mu_{j,k}\prod_{\ell=1}^j
c_{{\bf i}(\ell);k}(\ell-1)\prod_{\ell=j+1}^{r}d_{{\bf
i}(\ell);k}(\ell-1)\Big)\nonumber
\end{eqnarray}
(at the $\ell$-th step of the process we multiply the the preceding
equality by $v_{{\bf i}(\ell)}(\ell-1)$ on the left and by $d_{{\bf
i}(\ell);\,k}(\ell-1)$ on the right and then apply (\ref{16}) and
(\ref{17}) with $s=\ell$).

\medskip

\noindent
(B) In part~(A) we have introduced certain {\it nonzero}
elements
\begin{equation}\label{ele}
u_i,v_i, u_{i,j}, v_{i,j}, a_{i,j}, b_{i,j}, u_{i(1),\ldots,i(s)},
v_{i(1),\ldots,i(s)}, u_{{\bf i}(j)}(\ell), v_{{\bf
i}(j)}(\ell),c_{{\bf i}(j);\,k}(\ell), d_{{\bf i}(j);\,k}(\ell)
\end{equation}
in $U(\g)$ with $i,j,k, s,{\bf i}(j), \ell$ ranging over finite sets
of indices. These elements satisfy algebraic equations  (\ref{uv}),
(\ref{10}), (\ref{1'''}), (\ref{2'''}), (\ref{13}), (\ref{14}),
(\ref{16}) and (\ref{17}). We have also introduced, for $1\le k\le
r$, two {\it nonzero} finite collections of scalars
$\{\lambda_{j,k}(a_1,\ldots,a_l)\}$ and
$\{\mu_{j,k}(a_1,\ldots,a_l)\}$ in $\mathbb K$ linked with the
elements (\ref{ele}) by equation~(\ref{18}).

The procedure described in part~(A) shows that the above data can be
parametrised by the points of a locally closed subset of an affine
space ${\mathbb A}^D_{\mathbb K}$, where $D$ is sufficiently large.
More precisely, there exist finite sets $\mathcal F$ and $\mathcal
G$ of polynomials in $D$ variables with coefficients in $\mathbb Z$
such that a point $x\in{\mathbb A}^D_{\mathbb K}$ lies in our
locally closed set if and only if and $f(x)=0$ for all $f\in
\mathcal F$ and $g(x)\ne 0$ for some $g\in \mathcal G$. Let
$\widetilde{X}$ denote the zero locus of the set $\mathcal F$ in
${\mathbb A}^D_{\mathbb K}$.

Suppose the Gelfand--Kirillov conjecture holds for $\g$. Then there
exists $x\in \widetilde{X}(\mathbb K)$ such that $g(x)\ne 0$ for
some $g\in \mathcal G$. We set $X:=\{x\in\widetilde{X}\,|\,\,g(x)\ne
0\}$, a nonempty principal open subset of $\widetilde{X}$. As $X$ is
an affine variety defined over the algebraic closure
$\overline{\mathbb Q}$ of the field of rationals, we have that
$X(\overline{\mathbb Q})\ne\emptyset$. Hence there is a finitely
generated $\Z$-subalgebra $A$ of $\overline{\mathbb Q}$ for which
$X(A)\ne\emptyset$. There are an algebraic number field $K$ and a
nonzero $d\in\Z$ such that $A\subset {\mathcal O}_K[d^{-1}]$, where
${\mathcal O}_K$ denotes the ring of algebraic integers of $K$.
Since the map ${\rm Spec}({\mathcal O}_K)\to{\rm Spec}(Z)$ induced
by inclusion $\Z\hookrightarrow{\mathcal O}_K$ is surjective, it
must be that $X(\k)\ne\emptyset$ for every prime $p\in\mathbb N$
with $p\nmid d$ (recall that $\k$ stands for the algebraic closure
of ${\mathbb F}_p$).

\medskip

\noindent (C) When $X(\k)\ne\emptyset$, we can find {\it nonzero}
elements $$u_i,v_i, u_{i,j}, v_{i,j}, a_{i,j}, b_{i,j},
u_{i(1),\ldots,i(s)}, v_{i(1),\ldots,i(s)}, u_{{\bf i}(j)}(\ell),
v_{{\bf i}(j)}(\ell),c_{{\bf i}(j);\,k}(\ell), d_{{\bf
i}(j);\,k}(\ell)\in U(\g_\k)$$ satisfying (\ref{uv}), (\ref{10}),
(\ref{1'''}), (\ref{2'''}), (\ref{13}), (\ref{14}), (\ref{16}),
(\ref{17}) and {\it nonzero} collections of scalars
$\{\lambda_{j,k}(a_1,\ldots,a_l)\}$ and
$\{\mu_{j,k}(a_1,\ldots,a_l)\}$ in $\k$ for which the modular
version of (\ref{18}) holds. As all steps of the procedure described
in part~(A) are reversible and the nonzero elements of $U(\g_\k)$
still form an Ore set, this enables us to find $w_1,\ldots,
w_{2N}\in$ {\sl Frac}$\,\,U(\g_\k)$ and {\it nonzero} polynomials
$P_1,\ldots, P_n$ and $Q_1,\ldots, Q_n$ in $w_1,\ldots, w_{2N}$ with
coefficients in the invariant algebra $U(\g_\k)^{G_\k}$ for which
the modular versions of (\ref{1}), (\ref{2}) and (\ref{3}) hold.
Since the images of $x_1,\ldots, x_n$ in $\g_\k$ generate {\sl
Frac}$\,\,U(\g_\k)$ as a skew-field, applying [\cite{Bois},
Lem.~1.2.3] shows that the Gelfand--Kirillov conjecture holds for
$\g_\k$ for all $p\gg 0$.
\end{proof}

\section{\bf The Gelfand--Kirillov conjecture and purity of field extensions}\label{sec2}
\subsection{}
In this section we investigate the modular situation under the
assumption that $p\gg 0$. We are going to apply recent results of
Rudolf Tange [\cite{T}] on the Zassenhaus variety of $\g_\k$ to show
that if the Gelfand--Kirillov conjecture folds for $\g_\k$, then the
field $\k(\g_\k^*)\,=\,${\sl Frac}$\,\,S(\g_\k)$ is purely
transcendental over its subfield $\k(\g_\k^*)^{G_\k}\,=\, ${\sl
Frac}$\,\,S(\g_\k)^{G_\k}$. To explain Tange's results in detail we
need a geometric description of the Zassenhaus variety of $\g_\k$.
We follow the exposition in [\cite{T}] very closely.

Recall that the {\it Zassenhaus variety} $\mathcal Z$ of $\g_\k$ is
defined as the maximal spectrum of the centre $Z(\g_\k)$ of
$U(\g_\k)$. The Lie algebra $\g_\k=\Lie(G_\k)$ carries a natural
$p$-th power map $x\mapsto x^{[p]}$ equivariant under the adjoint
action of $G_\k$. We denote by $Z_p(\g_\k)$ the $p$-centre of
$U(\g_\k)$; it is generated as a $\k$-algebra by all
$\eta(x):=x^p-x^{[p]}$ with $x\in\g_\k$. It follows easily from the
PBW theorem that $Z_p(\g_\k)$ is a polynomial algebra in
$\eta(x_1),\ldots,\eta(x_n)$ and $U(\g_\k)$ is a free
$Z_p(\g_\k)$-module of rank $p^n$. This implies that $Z(\g_\k)$ is
is a Noetherian domain of Krull dimension $n=\dim\,\g_\k$, thus
showing that $\mathcal Z$ is an irreducible $n$-dimensional affine
variety. By an old result of Zassenhaus [\cite{Zass}], the variety
$\mathcal Z$ is normal.

\subsection{} To ease notation we often identify the elements of
$\g_\Z$ with their images in $\g_\k= \g_\Z\otimes_\Z\k$. Recall that
$\mathcal{B}=\{x_1,\ldots, x_n\}$ is a Chevalley basis of $\g_\Z$.
Then there is a maximal torus of $T\subset G$ defined and split over
$\mathbb Q$, such that
$$\mathcal{B}\,=\,\{h_\alpha\,|\,\,\alpha\in\Pi\}\cup
\{e_\alpha\,|\,\,\alpha\in\Phi\},$$ where $\Phi$ is the root system
of $G$ with respect to $T$ and $\Pi$ is a basis of simple roots in
$\Phi$ (we adopt the standard convention that $h_\alpha=({\rm
d}\alpha^\vee)(1)$ where ${\rm d}\alpha^\vee$ is the differential at
$1$ of the coroot $\alpha^\vee\colon\,\k^\times\to T$, and
$e_\alpha$ is a generator of the $\Z$-module $\g_\Z\cap\g_\alpha$,
where $\g_\alpha$ is the $\alpha$-root space of $\g$ with respect to
$T$).

Set $\t:=\Lie(T)$ and denote by $T_\k$ the maximal torus of $G_\k$
obtained from $T$ by base change. Set $\t_\k:={\rm Lie}(T_\k)$, and
identify the dual space $\t_\k^*$ with the subspace of $\g_\k^*$
consisting of all linear functions $\chi$ on $\g_\k$ with
$\chi(e_\alpha)=0$ for all $\alpha\in\Phi$. We write $X(T_\k)$ for
the group of rational characters of $T_\k$ and denote by $W$ the
Weyl group $N_{G_\k}(T_\k)/Z_{G_\k}(T_\k)$. This group is generated
by reflections $s_\alpha$ with $\alpha\in\Phi$ and it acts naturally
on both $\t_\k$ and $\t_\k^*$.

Let $\Phi_+$ be the positive system of $\Phi$ containing $\Pi$ and let
$\rho=\frac{1}{2}\sum_{\alpha\in\Phi_+}\alpha$. Then ${\rm d}\rho$ is and ${\mathbb F}_p$-linear
combination of the ${\rm d}\alpha$'s with $\alpha\in\Pi$. To ease notation we write $\rho$ instead
of ${\rm d}\rho$. The {\it dot action} of $W$ on $\t_\k^*$ is defined as follows:
$$w_{\,\bullet\,}\chi\,=\,(\chi+\rho)-\rho\qquad\qquad(\forall \, w\in W,\ \chi\in\t_\k^*).$$
The induced dot action of $W$ on $S(\t_\k)$ has the property that ${s_\alpha}_{\,\bullet}\,t=
s_\alpha(t)-\alpha(t)$ for all $t\in\t_\k$ and $\alpha\in\Pi$.
There exists a unique algebra isomorphism $\gamma\colon\,S(\t_\k)\stackrel{\sim}{\too}S(\t_\k)$ such that
$\gamma(t)=t-\rho(t)$ for all $t\in\t_\k$. The dot action of $W$ is related to natural action of $W$
on $S(\t_\k)$ by the rule $w_{\,\bullet}=\gamma^{-1}\circ w\circ \gamma$ for all $w\in W$, which gives
rise to an isomorphism of invariant algebras $\gamma\colon\,S(\t_\k)^{W_\bullet}\stackrel{\sim}{\too}
S(\t_\k)^W$.

Put $\Phi_-=-\Phi_+$ and write $\n_\k^\pm$ for the $\k$-span of the $e_\alpha$'s with
$\alpha\in\Phi_{\pm}$. Then
$S(\g_\k)\,=\, S(\n_\k^-)\otimes_\k S(\t_\k)\otimes_\k S(\n_\k^+)$ and
$U(\g_\k)\,=\, U(\n_\k^-)\otimes_\k U(\t_\k)\otimes_\k U(\n_\k^+)$ as vector
spaces. Write $S_+(\g_\k)$ and $U_+(\g_\k)$ for the augmentation ideals of $S(\g_\k)$ and $U(\g_\k)$,
respectively, and
denote by $\Psi$ (resp., $\tilde{\Psi}$)
the linear map from
$S(\g_\k)$ onto $S(\t_\k)$ (resp., from
$U(\g_\k)$ onto $U(\t_\k)=S(\t_\k)$) taking
$u\otimes h\otimes v$ with $u\in S(\n_\k^-)$, $h\in S(\t_\k)$, $v\in S(\n_\k^+)$
(resp., $u\in U(\n_\k^-)$, $h\in U(\t_\k)$, $v\in U(\n_\k^+)$) to $u_0hv_0$, where $x_0$ is the scalar part of
$x\in S(\g_\k)$ (resp.,
$x\in U(\g_\k)$) with respect to the decomposition
$S(\g_\k)\,=\,\k1\oplus S_+(\g_\k)$
(resp., $U(\g_\k)\,=\,\k1\oplus U_+(\g_\k)$). Note that
the map $\Psi$ is an algebra epimorphism and so is
the restriction of $\tilde{\Psi}$ to $U(\g_\k)^{T_\k}$.

For $g\in G_\k$, $x\in\g_\k$, $\chi\in\g_\k^*$ we write $g\cdot x$ for
$({\rm Ad}\,g)(x)$ and $g\cdot\chi$ for $({\rm Ad}^*\,g)(\chi)$.
Since $p\gg 0$, the Chevalley restriction theorem holds for $\g_\k$, that is, the restriction of $\Psi$ to
$S(\g_\k)^{G_\k}$ induces an isomorphism of invariant algebras
\begin{equation}\label{Chev}
\Psi\colon\,S(\g_\k)^{G_\k}\stackrel{\sim}{\too} S(\t_\k)^W.
\end{equation}
As $p$ is large, we can argue as in the proof of Proposition~2.1 in [\cite{Ve}]
to deduce that
the restriction of $\tilde{\Psi}$ to $U(\g_\k)^{G_\k}\subset U(\g_\k)^{T_\k}$ induces
an algebra isomorphism
\begin{equation}\label{HC}
\tilde{\Psi}\colon\,U(\g_\k)^{G_\k}\stackrel{\sim}{\too} S(\t_\k)^{W_{\,\bullet}}
\end{equation}
(in fact, this holds under very mild assumptions on $p$; see [\cite{KW}, Lem.~5.4]).

\subsection{}\label{r} As the Killing form $\kappa$ of $\g_\k$ is
nondegenerate for almost primes $p$, we may identify the
$G_\k$-modules $\g_\k$ and $\g_\k^*$ by means of Killing isomorphism
$\kappa\colon\,\g_\k\ni x\mapsto \kappa(x,\,\cdot\,\,)\in\g_\k^*$.
If $\chi=\kappa(x,\,\cdot\,\,)\in\g_\k^*$ and $x=x_s+x_n$ is the
Jordan--Chevalley decomposition of $x$ in the restricted Lie algebra
$\g_\k$, then we define $\chi_s:=\kappa(x_s,\,\cdot\,\,)$ and
$\chi_n:=\kappa(x_n,\,\cdot\,\,)$. We call $\chi_s$ and $\chi_n$ the
{\it semisimple} and {\it nilpotent} part of $\chi$. Denote by
$(\t_\k)_{\rm reg}$ the set of all regular elements of $\t$ and put
$(\t_\k^*)_{\rm reg}:= \kappa\big((\t_\k)_{\rm reg}\big)$. The
elements of $(\t^*_\k)_{\rm reg}$ are called {\it regular linear
functions} on $\t$. Note that $\chi\in(\t_\k^*)_{\rm reg}$ if and
only if $\chi=\kappa(t,\,\cdot\,\,)$ for some $t\in\t_\k$ whose
centraliser in $\g_\k$ equals $\t_\k$. It follows that
$\chi\in(\t_\k^*)_{\rm reg}$ if and only if $\chi(h_\alpha)\ne 0$
for all $\alpha\in\Phi$. In view of [\cite{St}, Cor.~2.6], this
implies that $\chi\in(\t_\k^*)_{\rm reg}$ if and only if the
stabiliser of $\chi$ in $W$ is trivial. As a consequence,
$Z_{G_\k}(\chi)=T_\k$ for every $\chi\in(\t_\k^*)_{\rm reg}$.

Denote by $(\g_\k)_{\rm rs}$ the set of all regular semisimple
elements of $\g_\k$. Since every semisimple element of $\g_\k$ lies
in the Lie algebra of a maximal torus of $G_\k$ and all maximal tori
of $G_\k$ are conjugate, we have the equality $(\g_\k)_{\rm
rs}\,=\,G_\k\cdot(\t_\k)_{\rm reg}$; see [\cite{Hu}, \S~13] or
[\cite{BoSp}, 4.5]. We set $(\g_\k^*)_{\rm rs}\,:=\,
\kappa\big(G_\k\cdot(\t_\k)_{\rm reg}\big)$ and call the elements of
$(\g_\k^*)_{\rm rs}$ {\it regular semisimple linear functions} on
$\g_\k$.

Now define $\bar{H}:=\prod_{\alpha\in\Phi}h_\alpha$, an element of
$S(\t_\k)^W$, and pick $H\in S(\g_\k)^{G_\k}$ such that
$\Psi(H)=\bar{H}$. It is well known (and easy to see when $p\gg 0$)
that for all $\chi\in\g_\k^*$ and $f\in S(\g_\k)^{G_\k}$ one has
$f(\chi)=f(\chi_s)$. As $(\g_\k^*)_{\rm rs}\,=\,
G_\k\cdot(\t_\k^*)_{\rm reg}$ and $\chi\in(\t^*_\k)_{\rm reg}$ if
and only if $\bar{H}(\chi)\ne 0$, the $G_\k$-conjugacy of maximal
toral subalgebras of $\g_\k$ implies that
\begin{equation}\label{rss}
(\g_\k^*)_{\rm rs}\,=\,\{\chi\in\g_\k^*\,|\,\,H(\chi)\ne 0\}
\end{equation} is a principal
Zariski open subset of $\g_\k^*$. The Weyl group $W$ acts on the
affine variety $(G_\k/T_\k)\times (\t_\k^*)_{\rm reg}$ by the rule
$w(gT_\k,\lambda)=(gw^{-1}T_\k,w(\lambda))$ and this action commutes
with the left regular action of $G_\k$ on the first factor. It
follows from [\cite{Bo}, Prop.~II.\,6.6 and Thm.~AG. 17.3] that the
coadjoint action-morphism gives rise to a $G_\k$-equivariant
isomorphism of affine algebraic varieties
$$\big((G_\k/T_\k)\times (\t_\k^*)_{\rm reg}\big)/W\stackrel{\sim}{\too}(\g_\k^*)_{\rm rs};$$
see [\cite{T}, 1.3] for more detail.
\subsection{} For a vector space $V$ over $\k$ the the {\it Frobenius
twist} $V^{(1)}$ is defined as the vector space over $\k$ with the
same underlying abelian group as $V$ and with scalar multiplication
given by $\lambda\cdot v:=\lambda^{1/p}v$ for all $v\in V$ and
$\lambda\in\k$. The polynomial functions on $V^{(1)}$ are the $p$-th
powers of those on $V$. The identity map $V\to V^{(1)}$ is a
bijective closed morphism of affine varieties, called the {\it
Frobenius morphism}. The image of a subset $Y\subseteq V$ under this
morphism is denoted by $Y^{(1)}$. The Frobenius twist of a
$\k$-algebra $V$ is defined similarly: the scalar multiplication is
modified as above, but the product in $V$ is unchanged. If $V$ has
an ${\mathbb F}_p$-structure and $G_\k$ acts on $V$ as algebra
automorphisms via a rational representation $\rho\colon\,
G_\k\to{\rm GL}(V)$ defined over ${\mathbb F}_p$, then $G_\k$ also
acts on $V^{(1)}$ (as algebra automorphisms) via the rational
representation $\rho\circ{\rm Fr}$, where ${\rm Fr}$ is the
Frobenius endomorphism of $G_\k$. This action coincides with the one
given by composing $\rho$ with the Frobenius endomorphism of ${\rm
GL}(V)$ associated by the ${\mathbb F}_p$-structure of $V$.

The preceding remark applies in the case where $V=S(\g_\k)$ and
$\rho\colon\,G_\k\to \mathrm{GL}(V)$ is the rational $G_\k$-action
by algebra automorphisms extending the adjoint action of $G_\k$. The
${\mathbb F}_p$-structure of $S(\g_\k)$ is given by the canonical
isomorphism $S(\g_\k)\cong S(\g_{{\mathbb F}_p})\otimes_{{\mathbb
F}_p}\k$ where $\g_{{\mathbb F}_p}=\g_\Z\otimes_{\Z}{\mathbb F}_p$.
Thus, there is a $\k$-algebra isomorphism $\phi\colon\,
S(\g_\k)^{(1)}\stackrel{\sim}{\to}S(\g_\k)$ such that $\phi(g\cdot
f)\,=\,(g^{\rm Fr})(\phi(f))$ for all $g\in G_\k$ and $f\in
S(\g_\k)^{(1)}$.

The rule $g\star f:=\phi^{-1}(g(\phi(f)))$ defines a rational action
of $G_\k$ on $S(\g_\k)^{(1)}=\, \k[(\g_\k^{(1)})^*]\cong\,
\k[(\g_\k^*)^{(1)})]$. In [\cite{T}], the induced action of $G_\k$
on $(\g_\k^*)^{(1)}$ is called the the {\it star action}. By
construction, it has the property that
\begin{equation}\label{star1}
g^{\rm Fr}\star\chi\,=\,g\cdot \chi\qquad\quad \ \ \,
\big(\forall\,g\in G_\k,\ \chi\in (\g^*_\k)^{(1)}\big).
\end{equation}

It was first observed in [\cite{Kry}] that the algebra map $\eta\colon\,S(\g_\k)^{(1)}=S(\g_\k^{(1)})\to\,
Z(\g_\k)$ sending $x\in\g_\k$ to $\eta(x)\in Z_p(\g_\k)$ is a $G_\k$-equivariant
algebra isomorphism. One checks easily that $\eta\circ\Psi=\tilde{\Psi}\circ \eta$. Also,
$\gamma(\eta(t))=\eta(t)$ for all $t\in\t_\k$, which stems from the fact that $\rho(t^{[p]})=\rho(t)^p$.

\subsection{} In [\cite{T}], Tange introduced a principal open subset
$\mathcal{Z}_{\rm rs}$ of $\mathcal{Z}$ and showed that  it is
isomorphic to a  principal open subset of $\g_\k^*$ contained in
$(\g^*_\k)_{\rm rs}$. In order to explain his construction in detail
we need a more explicit description of the variety $\mathcal{Z}$.

Recall from Sect.~\ref{sec1} that
$Z(\g_\k)^{G_\k}=\,U(\g_\k)^{G_\k}=\,\k[\overline{\psi}_1,\ldots,\overline{\psi}_l]$
is a polynomial algebra in $l$ variables, where
$\overline{\psi}_1,\ldots,\overline{\psi}_l$ are the images in
$U(\g_\k)$ of algebraically independent generators $\psi_1,\ldots,
\psi_l$ of $Z(\g)$ contained in $U(\g_\Z)$. In view of (\ref{HC})
and properties of $\gamma$, this implies that both
$S(\t_\k)^{W_{\,\bullet}}$ and $S(\t_\k)^{W}$ are polynomial
algebras in $l$ variables. It is worth mentioning that the map in
(\ref{Chev}) gives rise to a natural isomorphism
$S(\g_\k^{(1)})^{G_\k}\stackrel{\sim}{\too} S(\t_\k^{(1)})^W$.

By Veldkamp's theorem,
$$Z(\g_\k)\,\cong\, Z_p(\g_\k)\otimes_{Z_p(\g_\k)^{G_\k}}U(\g_\k)^{G_\k}$$
and, moreover, $Z_p(\g_\k)$ is a free $Z_p(\g_\k)$-module with basis
$\{\overline{\psi}_1^{\,a_1}\cdots\overline{\psi}_l^{\,a_l}\,|\,0\le
a_i\le p-1\}$; see [\cite{Ve}]. A geometric interpretation of
Veldkamp's theorem is given in [\cite{MR}]. Following [\cite{T},
1.6] we let $\xi\colon\,\t_\k^*\to(\t_\k^{(1)})^*$ be the morphism
induced by $\eta\colon\,S(\t_\k^{(1)})\to U(\t_\k)=S(\t_\k)$ and let
$\zeta\colon\,(\g_\k^{(1)})^*\to (\t_\k^{(1)})^*/W$ be the morphism
associated with the composite
$$\k[(\t_\k^{(1)})^*]^W\stackrel{\sim}{\too}\k[(\g_\k^{(1)})^*]^{G_\k}\hookrightarrow\,
\k[(\g_\k^{(1)})^*],$$ where the first isomorphism is induced by
$\Psi^{-1}$. Let $\pi\colon\,(\t_\k^{(1)})^*\to(\t_\k^{(1)})^*/W$
and $\pi_{\,\bullet}\colon\, \t_\k^*\to\t_\k^*/W_{\,\bullet}$ be the
quotient morphisms. Note that $\xi(\lambda)(t)=\lambda(t)^p-
\lambda(t^{[p]})$. If $\lambda$ lies in the ${\mathbb F}_p$-span of
$\Pi$, then $\lambda(t)^p=\lambda(t^{[p]})$ for all $t\in\t_\k$
because $h_\alpha^{[p]}=h_\alpha$ for all $\alpha\in\Phi$. Thus,
$\xi(\lambda)=0$ in that case. Applying this with $\lambda=\rho$, we
see that
$\xi(w_{\,\bullet}\,\lambda)=\xi(w(\lambda))=w(\xi(\lambda))$ for
all $t\in\t_\k$ and $w\in W$. Also, $\zeta(\chi)=\pi(\chi_s')$,
where $\chi_s'$ is a $G_\k$-conjugate of $\chi_s$ that lies in
$(\t_\k^{(1)})^*$ (it is important here that $\pi(\chi_s')$ is
independent of the choice of $\chi_s'$, which follows from the fact
that the intersection of $(\t_\k^{(1)})^*$ with $G_\k\cdot\chi$ is a
single $W$-orbit in $(\t_\k^{(1)})^*$). Finally, define
$\nu\colon\,(\g_\k^*)^{(1)}\stackrel{\sim}{\to}(\g_\k^{(1)})^*$ by
setting $\nu(\chi)=\chi^p$ for all $\chi\in(\g_\k^*)^{(1)}$.
 By [\cite{MR}, Cor.~3], there is a canonical $G_\k$-equivariant isomorphism
\begin{equation}\label{MR}
 \mathcal{Z}\stackrel{\sim}{\too}(\g_\k^*)^{(1)}\times_{(\t_\k^*)^{(1)}/W}\,\t^*_\k/W_{\,\bullet}
 \end{equation}
 where the $G_\k$-action on the fibre product is given by from the coadjoint action
 on the first factor, the morphism $\t_\k^*/W_{\,\bullet}\too(\t_\k^{(1)})^*/W$ is induced by $\xi$
 and the morphism $(\g^*)^{(1)}\to(\t_\k^{(1)})^*/W$ is the composite of $\nu$ and $\zeta$.

 \subsection{} In what follows we identify $\mathcal{Z}$ with a closed subset of the affine space
 $(\g^*_\k)^{(1)}\times\t_\k^*/W_{\,\bullet}$ by means of isomorphism (\ref{MR}).
Note that
$(\chi,\pi_{\,\bullet}(\lambda))\in(\g^*_\k)^{(1)}\times\t_\k^*/W_{\,\bullet}$
belongs to $\mathcal{Z}$ if and only if there exists $w\in W$ such
that $$\lambda(t)^p-\lambda(t^{[p]}) \,=\,w(\chi'_s)^p\qquad\quad\ \
(\forall\, t\in\t_\k)$$ where $\chi_s'\in \,\t_\k^*\cap
(G_\k\cdot\chi_s)$.

Recall that $G_\k$ operates on $(\g_\k^*)^{(1)}$ via the star action (\ref{star1}). From the above discussion
it follows that this action gives rise to the star action on the Zassenhaus variety $\mathcal{Z}$ via:
\begin{equation}\label{star2}
g\star(\chi,\pi_{\,\bullet}(\lambda)):=(g\star\chi,\pi_{\,\bullet}(\lambda))\qquad\quad\
\ \big(\forall\,\,
(\chi,\pi_{\,\bullet}(\lambda))\in\mathcal{Z}\big).
\end{equation}
Following [\cite{T}, Sect.~2], we now define $\mathcal{Z}_{\rm
rs}\,:=\,{\rm pr}_1^{-1} \big((\g_{\rm rs}^*)^{(1)}\big)$, where
${\rm pr}_1\colon\,\mathcal{Z}\to (\g_\k^*)^{(1)}$ is the first
projection. In view of (\ref{rss}) it is straightforward to see that
$$\mathcal{Z}_{\rm rs}\,=\,\big\{(\chi,\pi_{\,\bullet}(\lambda))\in\mathcal{Z}\,|\,\,H^p(\chi)\ne 0\big\}$$
is a nonempty principal open subset of $\mathcal{Z}$.

Set $\bar{F}:=\prod_{\alpha\in\Phi}\,(h_\alpha^p-h_\alpha)$, an
element of $S(\t_\k)^W$, and pick $F\in S(\g_\k)^{G_\k}$ with
$\Psi(F)=\bar{F}$. Note that $H\mid F$ because $\bar{H}$ divides
$\bar{F}$. Define
 $$(\t_\k^*)'_{\rm rs}\,:=\,\{\chi\in\t_\k^*\,|\,\,\bar{F}(\chi)\ne 0\}\ \ \mbox{ and
 }\ \
(\g_\k^*)'_{\rm rs}\,:=\,\{\chi\in\g_\k^*\,|\,\,F(\chi)\ne 0\}.$$
Clearly, $(\t_\k^*)'_{\rm rs}$ consists of all $\chi\in\t_\k^*$ with
$\chi(h_\alpha)\not\in{\mathbb F}_p$ for all $\alpha\in\Phi$. The
preceding remark shows that $(\g_\k^*)'_{\rm rs}$ is a principal
open subset of $\g_\k^*$ contained in the principal open set
$(\g_\k^*)_{\rm rs}\,=\,G_\k\cdot(\t_\k^*)_{\rm rs}$. Therefore,
$(\g_\k^*)_{\rm rs}'\,= \,G_\k\cdot(\t_\k^*)'_{\rm rs}$. By
[\cite{T}, Thm.~1], there is an isomorphism of algebraic varieties
$\beta\colon\, \mathcal{Z}_{\rm
rs}\stackrel{\sim}{\too}(\g_\k^*)_{\rm rs}'$ which intertwines the
star action of $G_\k$ on $\mathcal{Z}$ with the coadjoint action in
the following sense:
\begin{equation}\label{inter}
\beta(g\star(\chi,\pi_{\,\bullet}(\lambda)))=g\cdot
\beta((\chi,\pi_{\,\bullet}(\lambda)))\qquad\quad\ \,
\big(\forall\,\,(\chi,\pi_{\,\bullet}(\lambda)) \in\mathcal{Z}_{\rm
rs}\big).
\end{equation}
\subsection{} For $1\le i\le l$ set $\varphi_i:={\rm gr}\,\psi_i$, a
homogeneous element of $S(\g_\Z)={\rm gr}\, U(\g_\Z)$. Since $p\gg
0$, we may also assume that the elements
$\varphi_1,\ldots,\varphi_l$ generate the invariant algebra
$S(\g)^\g$ and their images
$\overline{\varphi}_1,\ldots,\overline{\varphi}_l$ in
$S(\g_\k)=S(\g_\Z)\otimes_\Z\k$ generate $S(\g_\k)^{G_\k}$.

\medskip

We are now ready to prove the main result of this section:
\begin{theorem}\label{pur}
If the Gelfand--Kirillov conjecture holds for $\g$, then for all
$p\gg 0$ the field of rational functions $\k(\g_\k^*)=${\sl
Frac}$\,\,S(\g_\k)$ is purely transcendental over its subfield
$k(\g_\k^*)^{G_\k}=\,k(\overline{\varphi}_1,\ldots,\overline{\varphi}_l)$.
\end{theorem}
\begin{proof}
Suppose the Gelfand--Kirillov conjecture holds for $\g$.
Theorem~\ref{GK mod p} then says that it holds for $\g_\k$ for all
$p\gg 0$. More precisely, it follows from the proof of
Theorem~\ref{GK mod p} that $\mathcal{D}(\g_\k)$ is generated as a
skew-field by $\overline{\psi}_1,\ldots,\overline{\psi}_l\in
Z(\g_\k)$ and elements $w_1,\ldots, w_{2N}$ which satisfy relations
(\ref{1}) and (\ref{2}) (here $N=|\Phi_+|$). For $1\le i\le 2N$ set
$z_i:=w_i^p$. Since ${\mathcal D}(\g_\k)\,\cong
\mathcal{D}_{N,\,l}(\k)$ as $\k$-algebras, the elements $z_1,\ldots,
z_{2N}$ are central in {\sl Frac}$\,\,U(\g_\k)$. Moreover, the
centre of ${\mathcal D}(\g_\k)$ is $\k(z_1,\ldots,
z_{2N},\overline{\psi}_1,\ldots,\overline{\psi}_l)$ and the elements
$z_1,\ldots, z_{2N},\overline{\psi}_1,\ldots,\overline{\psi}_l$ are
algebraically independent; see [\cite{Bois}, 1.1.3] for more detail.

 On the other hand, it is well known that in the modular case ${\mathcal D}(\g_\k)$ is the central
 localisation
 of $U(\g_\k)$ by the set $Z_p(\g_\k)^\times$ of nonzero elements of
 $Z_p(\g_\k)$. Likewise, the centre of ${\mathcal D}(\g_\k)$ is the
 localisation of $Z(\g_\k)$ by the set $Z_p(\g_\k)^\times$. It follows that
 the centre of ${\mathcal D}(\g_\k)$ equals $\Q(\g_\k)=\Q_p[\overline{\psi}_1,\ldots,\overline{\psi}_l]$,
 where $\Q_p$ is the field of fractions of $Z_p(\g_\k)$.
 Since the $\Q_p$-vector space $\Q(\g_\k)$ has a basis consisting of monomials in
 $\overline{\psi}_1,\ldots,\overline{\psi}_l$, it is straightforward to see that
 the field of invariants $\Q(\g_\k)^{G_\k}$ coincides with $\Q_p^{G_\k}[\overline{\psi}_1,\ldots,\overline{\psi}_l]$.
 As $Z_p(\g_\k)$ is a polynomial algebra and the
 connected group $G_\k$ coincides with its derived subgroup, we have that
 $\Q_p^{G_\k}\,=\,${\sl Frac}$\,\,Z_p(\g_\k)^{G_\k}$. This shows that
 $\Q(\g_\k)^{G_\k}\,=\,${\sl Frac}$\,\,Z(\g_\k)^{G_\k}\,=\,\k(\overline{\psi}_1,\ldots,\overline{\psi}_l)$.
 We thus deduce that
\begin{equation}\label{fields}
 \k(\mathcal{Z})=\,\Q(\g_\k)\,=\,\k(z_1,\ldots, z_{2N},\overline{\psi}_1,\ldots,\overline{\psi}_l)\,=\,
 \k(\mathcal{Z})^{G_\k}\big(z_1,\ldots,z_{2N}\big)
 \end{equation}
 is purely transcendental over the field of invariants
 $\k(\mathcal{Z})^{G_\k}\,=\,\k(\overline{\psi}_1,\ldots,\overline{\psi}_l)$.

 Recall that in our geometric realisation (\ref{MR}) the ordinary action of $G_\k$ on $\mathcal{Z}$
is given by $g\cdot(\chi,\pi_{\,\bullet}(\lambda))=(g\cdot\chi,
\pi_{\,\bullet}(\lambda))$ for all $g\in G_\k$ and all
$(\chi,\pi_{\,\bullet}(\lambda))\in\mathcal{Z}$. Since in (\ref{MR})
we regard $\chi$ as an element of $(\g_\k^*)^{(1)}$, comparing this
with (\ref{star1}) and (\ref{star2}) yields that every orbit with
respect to the ordinary action of $G_\k$ on $\mathcal{Z}$ is an
orbit of $G_\k$ with respect to the star action and vice versa. From
this it follows that both actions have the same rational invariants.

The comorphism of $\beta^{-1}\colon\,(\g_\k^*)_{\rm
rs}'\stackrel{\sim}{\too}\mathcal{Z}_{\rm rs}$ induces a field
isomorphism between $\k(\mathcal{Z})$ and $\k(\g_\k^*)$; we call it
$b$. Combining (\ref{inter}) with the preceding remark one observes
that $b$ sends the subfield
$\k(\mathcal{Z})^{G_\k}=\,\k(\overline{\psi}_1,\ldots,
\overline{\psi}_l)$ onto $\k(\g_\k^*)^{G_\k}$. But then
(\ref{fields}) shows that
$\k(\g^*_\k)\,=\,\k(\g_\k^*)^{G_\k}\big(b(z_1),\ldots,
b(z_{2N})\big)$ is purely transcendental over
$\k(\g_\k^*)^{G_\k}=\,\k(\overline{\varphi}_1,\ldots,\overline{\varphi}_l)$.
This completes the proof.
 \end{proof}
\begin{rem}\label{r2}
Combining Theorem~\ref{pur} with the Killing isomorphism
$\kappa\colon\,\g_\k\stackrel{\sim}{\to}\g_\k^*$ we see that for all
$p\gg 0$ the field of rational functions $\k(\g_\k)$ is purely
transcendental over its subfield $k(\g_\k)^{G_\k}$.
\end{rem}
\section{\bf Purity, generic tori and base change}
\subsection{}\label{4.1}  We keep the notation introduced in
Sections~\ref{sec1} and \ref{sec2} and assume that ${\rm
char}(\k)=p\gg 0$. Recall that
$\{x_1,\ldots,x_n\}\,=\,\{h_\alpha\,|\,\,\alpha\in\Pi\}\cup\{e_\alpha\,|\,\,\alpha\in\Phi\}$
is a Chevalley basis of $\g_\Z$ and we identify the $x_i$'s with
their images in $\g_\k$. Write $\Pi=\{\alpha_1,\ldots,\alpha_l\}$
and let $\{X_1,\ldots, X_l\}$ and
$\{X_{\alpha}\,|\,\,\alpha\in\Phi\}$ be two sets of independent
variables. Set $K:=\mathbb{Q}(X_1,\ldots,X_l)$ and
$\widetilde{K}:=K(X_\alpha\,|\,\,\alpha\in\Phi)$ and denote by $K_p$
an algebraic closure of $\k(X_1,\ldots,X_l)$. Write
$\widetilde{K}_p:=K_p(X_\alpha\,|\,\,\alpha\in\Phi)$ and denote by
$\mathbb{K}_p$ an algebraic closure of $\widetilde{K}_p$. To ease
notation we shall assume that $\mathbb K$ is an algebraic closure of
$\widetilde{K}$ (this will cause no confusion).

Given a field $F$ we write $\g_F$ for the Lie algebra
$\g_\Z\otimes_\Z F$ over $F$ and denote by $G_F$ the simple, simply
connected algebraic $F$-group with Lie algebra $\g_F$. Let
$\widetilde{t}:=\sum_{i=1}^lX_ih_{\alpha_i}$ and
$\widetilde{x}:=\sum_{\alpha\in\Phi}X_\alpha e_\alpha$. Since the
$X_i$'s are algebraically independent, $\widetilde{t}$ is a regular
semisimple element of $\g_{K}$ contained in $\g_\Z[X_1,\ldots,
X_l]$. Its image $\widetilde{t}_p:=\widetilde{t}\otimes 1\in
\g_\Z[X_1,\ldots,X_l]\otimes_Z\k$ is a regular semisimple element of
$\g_{K_p}$. The image of $\widetilde{x}$ in
$\g_\Z[X_\alpha\,|\,\,\alpha\in\Phi]\otimes_\Z\,\k$ is denoted by
$\widetilde{x}_p$. Set $\widetilde{y}:=\widetilde{t}+\widetilde{x}$
and $\widetilde{y}_p:=\widetilde{t}_p+\widetilde{x}_p$. These are
regular semisimple elements of $\g_{\widetilde{K}}$ and
$\g_{\widetilde{K}_p}$, respectively.

Write $G_p$ for the group $G_{\mathbb{K}_p}$ and $\g_p$ for its Lie
algebra $\g_{\mathbb{K}_p}$. Given a closed subgroup $H$ of $G_p$
defined over $\widetilde{K}_p$ we write $H_p$ for the group
$H(\mathbb{K}_p)$. Set $T^{\rm gen}:=Z_{G}(\widetilde{y})$ and
$T^{\rm gen}_p:=Z_{G_p}(\widetilde{y}_p)$. It follows from
[\cite{BoSp}, 4.3] that $T^{\rm gen}$ and $T^{\rm gen}_p$ are
maximal tori of $G$ and $G_p$ defined over $\widetilde{K}$ and
$\widetilde{K}_p$, respectively. Let $\t^{\,\rm gen}:=\Lie(T^{\,\rm
gen})$ and $\t^{\,\rm gen}_p:=\Lie(T^{\,\rm gen}_p)$.

\noindent \subsection{}\label{4.2} Let $\mathcal{T}_p$ be the
variety of maximal toral subalgebras of $\g_{p}$. As all maximal
toral subalgebras of $\g_{p}$ are conjugate under $G_{p}$ and the
normaliser $N_{p}$ of $\t_{p}:=\t_\k\otimes_\k \mathbb{K}_p$ in
$G_{p}$ is a reductive group, $\mathcal{T}_p\cong G_{p}/N_{_p}$ is
an affine algebraic variety. It follows from a well known result of
Grothendieck [\cite{DG}, Exp.~XIV, Thm.~6.2] that the variety
$\mathcal{T}_p$ is $K_p$-rational. More precisely, let $\m_p$ be
orthogonal complement to $\t_p$ with respect to the Killing form of
$\g_{K_p}$. A natural $K_p$-defined birational isomorphism between
$\mathcal{T}_p$ and $\m_p$ can be obtained as follows; see
[\cite{BoSp}, 7.9]:

 The set $\mathcal{T}_p^{\,\circ}$
of all $\h\in\mathcal{T}_p$ with $\h\cap\m_p=0$ is open, nonempty in
$\mathcal{T}_p$ and the set
$$\m_p^{\,\circ}\,:=\,\{m\in \m_p\,|\,\,
\widetilde{t}_p+m\in(\g_p)_{\rm rs}\,\, \mbox{ and }\,\, {\rm
ad}(\widetilde{t}_p+m)_{\vert\,\m_p} \,\mbox{ is injective}\,\}$$ is
open, nonempty in $\m_p$. For every $m\in\m_p^{\,\circ}$ the
centralizer of $\widetilde{t}_p+m$ in $\g_p$ is an element of
$\mathcal{T}_p^{\,\circ}$. Since for every
$\h\in\mathcal{T}_p^{\,\circ}$ there exists a unique
$m=m(\h)\in\m_p$ with $\widetilde{t}_p+m\in\h$, the map
$\mu\colon\,\mathcal{T}_p^{\,\circ}\too\m_\p^{\,\circ},\ \,\h\mapsto
m(\h),$ gives rise to a $K_p$-defined birational isomorphism between
$\mathcal{T}_p$ and $\m_p$. The $K_p$-defined birational map $\mu$
enables us to identify the field $K_p(\mathcal{T}_p)$ with
$K_p(\m_p)\,\cong\,
K_p(X_\alpha\,|\,\,\alpha\in\Phi)\,=\,\widetilde{K}_p$. It is
straightforward to see that $\widetilde{x}_p\in\m_p^\circ$ and
$\mu(\t^{\,\rm gen}_p)=
\widetilde{t}_p+\widetilde{x}_p=\widetilde{y}_p$. Since the field
$K_p(\widetilde{x}_p)=K_p(\widetilde{y}_p)$ is nothing but
$\widetilde{K}_p$, we now deduce that $\t^{\,\rm gen}_p$ is a
generic point of the $K_p$-variety $\mathcal{T}_p$.

\subsection{}\label{4.3}
Recall that $\varphi_1,\ldots,\varphi_l$ are free generators of
$S(\g)^\g$ contained in $S(\g_\Z)$ and such that
$S(\g_\k)^{G_\k}\,=\,\k[\overline{\varphi}_1,\ldots,
\overline{\varphi}_l]$, where $\overline{\varphi}_i=\varphi_i\otimes
1\in S(\g_\Z)\otimes_\Z\k=\,S(\g_\k)$. We identify the
$G_\k$-modules $\g_\k$ and $\g_\k^*$ by means the Killing
isomorphism $\kappa$; see Remark~\ref{r2}. Thus, we may regard
$\overline{\varphi}_1,\ldots,\overline{\varphi}_l$ as free
generators of the invariant algebra $\k[\g_\k]^{G_\k}$.

Let $Y_p$ be the fibre $\overline{\varphi}^{\,-1}(\widetilde{y}_p)$
of the adjoint quotient map
$\overline{\varphi}\colon\,\g_{p}\too\g_{p}/\!\!/G_{p},$ that is,
$$Y_p\,:=\,\{y\in\g_{p}\,|\,\,\overline{\varphi}_i(y)=
\overline{\varphi}_i(\widetilde{y}_p)\ \mbox{ for all }\, 1\le i\le
l\}.$$  As $p\gg0$, all fibres of $\overline{\varphi}$ are
irreducible complete intersections of dimension $n-l$ in the affine
space $\g_{p}$; see [\cite{Ve}] for more detail. Since
$\widetilde{y}_p$ is regular semisimple, the orbit
$G_{p}\cdot\widetilde{y}_p$ is Zariski closed in $\g_{p}$ and dense
in $Y_p$. This shows that $Y_p\,=\,G_{p}\cdot\widetilde{y}_p$ is a
smooth variety and the defining ideal of $Y_p$ is generated by the
regular functions $\overline{\varphi}_1-
\overline{\varphi}_1(\widetilde{y}_p),\ldots,\overline{\varphi}_l-
\overline{\varphi}_l(\widetilde{y}_p)$. Since $\widetilde{y}_p$ is
regular semisimple, the orbit map $G_{p}\to Y_p$ is separable.
Applying [\cite{Bo}, Prop.~II.\,6.6], we now deduce that the
$\widetilde{K}_p$-varieties $G_{p}/T^{\rm gen}_p$ and $Y_p$ are
$\widetilde{K}_p$-isomorphic (recall from (\ref{4.1}) that $T^{\rm
gen}_p$ is the centraliser of $\widetilde{y}_p$ in $G_{p}$).

\subsection{}\label{4.4}
Our next result is inspired by [\cite{CT}, Thm.~4.9]. The argument
in [\cite{CT}] exploits the notion of versality of
$(G,S)$-fibrations introduced in [\cite{CT}, Sect.~3] and seems to
rely on the characteristic zero hypothesis (see the footnote on
p.~20 of [\cite{CT}]). Our argument is different and it works under
very mild assumptions on the characteristic of the base field.
\begin{prop}\label{gen}
If the field $\k(\g_\k^*)$ is purely transcendental over
$\k(\g_\k^*)^{G_\k}$, then the homogeneous space $G_{p}/T_p^{\,{\rm
gen}}$ is $\widetilde{K}_p$-birational to an affine space.
\end{prop}
\begin{proof}
If $\k(\g_\k^*)$ is purely transcendental over $\k(\g_\k^*)^{G_\k}$,
then there exist $F_1,\ldots, F_{2N}\in\k(\g_\k)$ such that
$\k(\g)=\k(F_1,\ldots,
F_{2n},\overline{\varphi}_1,\ldots,\overline{\varphi}_l)$. Then the
rational map $F\colon\,\g_\k\dashrightarrow {\mathbb
A}_\k^{2N}\times{\mathbb A}_\k^l$ taking $y\in\g_\k$ to $F(y):=
\big((F_1(y),\ldots,F_{2N}(y),\,(\overline{\varphi}_1(y),\ldots,\overline{\varphi}_l(y)\big)\in
{\mathbb A}_\k^{2N}\times{\mathbb A}_\k^l$ induces a
$\k$-isomorphism $F\colon\,U\stackrel{\sim}{\to} V$ between a
$\k$-defined nonempty open subset $U$ of $\g_\k$ and a $\k$-defined
nonempty open subset $V$ of ${\mathbb A}_\k^{2N}\times{\mathbb
A}_\k^{l}$.

Since $f(\widetilde{y}_p)\ne 0$ for every nonzero $f\in\k[\g_\k]$
and $U=\g_\k\setminus Z$ for some Zariski closed $Z\subsetneq \g_\k$
defined over $\k$, we see that $\widetilde{y}_p\in
U_{p}:=\,\g_{p}\setminus Z(\mathbb{K}_p)$. Likewise,
$V=\big({\mathbb A}_\k^{2N}\times{\mathbb A}_\k^l\big)\setminus Z'$
for some Zariski closed subset $Z'\subsetneq {\mathbb
A}_\k^{2N}\times{\mathbb A}_\k^l$ defined over $\k$. We set
$V_{p}:=\,\big({\mathbb A}_{\mathbb{K}_p}^{2N}\times{\mathbb
A}_{\mathbb{K}_p}^l\big)\setminus Z'(\mathbb{K}_p)$ and observe that
$F$ gives rise to a $\k$-defined isomorphism between $U_{p}$ and
$V_{p}$.

Put $Y^{\circ}_p:=Y_p\cap U_{p}$. As $\widetilde{y}_p\in
Y^{\circ}_p$, we see that $Y^{\circ}_p$ is a nonempty closed subset
of $U_{p}$ defined over $\widetilde{K}_p$. Furthermore, $\dim
Y^{\circ}_p=n-l=2N$. Therefore, $F(Y^{\circ}_p)$ is a
$2N$-dimensional nonempty closed subset of $V_{p}$. On the other
hand, it is immediate from the definition of $F$ and our discussion
in (\ref{4.3}) that  $F(Y_p^\circ)\subseteq {\mathbb
A}^{2N}_{\mathbb{K}_p}\times {\rm pt}$. This implies that
$F(Y_p^\circ)$ is $\widetilde{K}_p$-isomorphic to a Zariski open
subset of ${\mathbb A}^{2N}_{\mathbb{K}_p}$ defined over
$\widetilde{K}_p$. Since $Y_p$ is $\widetilde{K}_p$-isomorphic to
$G_{p}/T_p^{\,{\rm gen}}$ by our discussion in (\ref{4.3}) and
$F(Y_p^\circ)$ is $\widetilde{K}_p$-isomorphic to $Y_p^\circ$ , we
conclude that the homogeneous space $G_{p}/T_p^{\,{\rm gen}}$ is
rational over $\widetilde{K}_p$.
\end{proof}
\subsection{} In order to adapt the proof of the crucial Theorem~6.3 from [\cite{CT}]
to our modular setting we need a smooth projective model of
$G_{p}/T_p^{\,{\rm gen}}$ defined over $\widetilde{K}_p$, that is, a
smooth projective $\widetilde{K}_p$-variety $Y^c_p$ together with an
open embedding $G_{p}/T_p^{\,{\rm gen}}\hookrightarrow Y^c_p$
defined over $\widetilde{K}_p$.
\begin{prop}\label{proj}
For all $p\gg 0$ the variety $G_{p}/T_p^{\,{\rm gen}}$ has a smooth
projective model defined over $\widetilde{K}_p$.
\end{prop}
\begin{proof}
Let $\varphi\colon\,\g\to\g/\!\!/ G$ be the adjoint quotient map and
set $Y:=\varphi^{-1}(\widetilde{y})$. Arguing as in (\ref{4.3}) we
observe that $Y=G\cdot\widetilde{y}$ is a smooth variety and the
defining ideal of $Y$ is generated by the regular functions
$\varphi_i-\varphi_i(\widetilde{y})$, where $1\le i\le l$. Our
discussion in (\ref{4.1}), (\ref{4.2}) and (\ref{4.3}) now shows
that there exists a finitely generated $\Z$-subalgebra $R$ of
$\widetilde{K}=K(\mathcal{T})$ and an affine flat scheme $\mathcal
Y$ of finite type over $S:=\,{\rm Spec}(R)$ such that
$$Y=\,{\mathcal Y}\times_{S}\,{\rm Spec}(\mathbb{K})\  \,\mbox{
and }\ Y_p=\,{\mathcal Y}\times_{S}\,{\rm Spec}(\mathbb{K}_p)\qquad\
(\forall\,p\gg 0).$$

By Hironaka's theorem on resolution of singularities there exists a
smooth projective $K(\mathcal{T})$-variety $Y^c\subseteq {\mathbb
P}^d_{\mathbb K}$ and an open immersion $\omega\colon\, Y\to Y^c$
defined over $K(\mathcal{T})$.  Let $\Gamma_\omega$ denote the graph
of $\omega$. Since all projective schemes are separated,
$\Gamma_\omega\,=\,\{(y,\omega(y))\,|\,\,y\in  Y\}$ is a closed
subset of $Y\times {\mathbb P}_{\mathbb K}^d$; see [\cite{Sh2},
p.~47]. As $K(\mathcal{T})$ is a perfect field, $\Gamma_\omega$ is
defined over $K(\mathcal{T})$; see [\cite{Bo}, AG, \S~14] for
detail.

Let $\widetilde{R}$ be a finitely generated $\Z$-subalgebra of
$K(\mathcal{T})$ containing $R$ and all elements which we need to
define $\omega$, $Y^c$, $\Gamma_\omega$, and the field isomorphism
$K(Y)\stackrel{\sim}{\to} K(Y^c)$ induced by the rational inverse of
the comorphism $\omega^*$. Then we obtain a projective scheme
${\mathcal Y}^c$ of finite type over $\widetilde{S}:=\,{\rm
Spec}(\widetilde{R})$ and an $\widetilde{S}$-morphism
$\widetilde{\omega}\colon\, \mathcal{Y}\to\mathcal{Y}^c$ whose base
change to ${\rm Spec}(\widetilde{K})$ is $\omega\colon\,Y\to Y^c$.
We also obtain an $\widetilde{S}$-subscheme
$\widetilde{S}$-subscheme $\Gamma_{\widetilde{\omega}}$ of
${\mathcal Y}\times_{\widetilde{S}} {\mathbb P}^d_{\widetilde{R}}$
such that
$\Gamma_\omega=\,\Gamma_{\widetilde{\omega}}\times_{\widetilde{S}}\,{\rm
Spec}(K(\mathcal{T})). $ By localising further as necessary we may
assume that the scheme $\mathcal{Y}^c$ is smooth over
$\widetilde{S}$ and the schemes $\mathcal Y$, ${\mathcal Y}^c$ and
$\Gamma_{\widetilde{\omega}}$ are flat over $\widetilde{S}$. We let
$\pi\colon \Gamma_{\widetilde{\omega}}\to \mathcal{Y}$ denote the
first projection.

Given a closed point $s\in \widetilde{S}$ and an
$\widetilde{S}$-scheme $\mathcal V$ we write $\kappa(s)$ for the
residue field of the local ring of $s$ and ${\mathcal V}_s$ for the
scheme-theoretical fibre $\mathcal{V}\times_{\widetilde{S}}\,{\rm
Spec}(\kappa(s))$. It follows from the above discussion that for
every closed point $s\in\widetilde{S}$ the schemes ${\mathcal Y}_s$
and $\mathcal{Y}^c_s$ are smooth and the base change
$\widetilde{\omega}_s\colon\,\mathcal{Y}_s \to\mathcal{Y}^c_s$ is
birational.

If $A$ is the affine coordinate ring of $\mathcal Y$, then
$\Gamma_\omega\subseteq \mathcal{Y}\times
\mathbb{P}^d_{\widetilde{R}}$ corresponds to a graded $A$-algebra
$B=B_0\oplus B_1\oplus B_2\oplus\ldots$ with $B_0=A$ generated over
$A$ by $d+1$ elements. By [\cite{Eis}, Thm.~14.8], there is an ideal
$J$ of $A$ such that for every prime ideal $P$ of $A$ the algebra
{\sl Frac}$\,(A/P)\otimes_{A}B$ has positive Krull dimension if and
only if $P\supseteq J$. We denote by $\mathcal{Y}'$ the closed
subscheme of $\mathcal{Y}$ corresponding to the ideal $J$. Then for
every closed point $s\in\widetilde{S}$  we have that $x\in
\mathcal{Y}'_s$ if and only if the fibre $\pi_s^{-1}(x) $ of the
base change $\pi_s\colon\,(\Gamma_{\widetilde{\omega}})_s\to
\mathcal{Y}_s$ has positive dimension.

Set $Y'\,:=\,\mathcal{Y}'\times_{\widetilde{S}}\,{\rm
Spec}(\mathbb{K})$. Since $\Gamma_\omega$ is closed and $\omega$ is
injective, the set $Y'_{\rm red}(\mathbb{K})$ is empty. In
conjunction with Hilbert's Nullstellensatz this implies that the
ideal $J\otimes_{\widetilde{R}} \widetilde{K}$ of
$A\otimes_{\widetilde{R}}\widetilde{K}=\,\widetilde{K}[Y]$ coincides
with $\widetilde{K}[Y]$. Then $\sum_{i=1}^kc_iq_i=1$ for some
$q_1,\ldots, q_k\in J$ and $c_1,\ldots,c_k\in\widetilde{K}=\,\,${\sl
Frac}$\,\,\widetilde{R}$. Localising $\widetilde{R}$ further, we may
assume that all $c_i$'s are in $\widetilde{R}$. Then the above
discussion shows that for every closed point $s\in\widetilde{S}$ the
reduced fibres of $\widetilde{\omega}_s\colon\,\mathcal{Y}_s\to
\mathcal{Y}^c_s$ are finite.

Since $\widetilde{R}$ is a Noetherian domain whose field of
fractions is $\widetilde{K}\,=\,K(X_\alpha\,|\,\,\alpha\in\Phi)$,
for every $p\gg 0$ there exists $s\in{\rm Spec}(\widetilde{R})$ with
$\kappa(s)\,=\,{\mathbb F}_p(X_1,\ldots,
X_l)\big(X_\alpha\,|\,\,\alpha\in\Phi\big)$. The discussion in
(\ref{4.3}) shows that for each such $s$ the scheme
$\mathcal{Y}_s\times_{{\rm Spec}(\kappa(s))}\,{\rm
Spec}(\mathbb{K}_p)$ is nothing but $Y_p$. Since $Y_p$ is reduced,
the base change
$\widetilde{\omega}_s\colon\,\mathcal{Y}_s\to\mathcal{Y}_s^c$ gives
rise to a natural morphism $\omega_p\colon\,Y_p\to
\big(\mathcal{Y}_s^c\times_{{\rm Spec}(\kappa(s))}\,{\rm
Spec}(\mathbb{K}_p)\big)_{\rm red}$. We denote by $Y_p^c$ the
irreducible component of the reduced scheme
$\big(\mathcal{Y}_s^c\times_{{\rm Spec}(\kappa(s))}\,{\rm
Spec}(\mathbb{K}_p)\big)_{\rm red}$ that contains $\omega_p(Y_p)$.

Since $\mathcal{Y}_s^c$ is smooth, projective, so too is $Y_p^c$.
Furthermore, our earlier remarks in the proof imply that $\omega_p$
is a $\widetilde{K}_p$-defined birational morphism of algebraic
varieties and all fibres of $\omega_p$ are finite. The variety
$Y_p^c$ is smooth, hence normal. Applying Zariski's Main Theorem to
the quasi-finite birational morphism $\omega_p\colon\,Y_p\to Y_p^c$,
we now deduce that $\omega_p$ is an open embedding; see [\cite{N},
Cor.~1(i)] for the statement and a short proof of the result we need
(it is worth mentioning that [\cite{N}] is available on the web).

Since the variety $Y_p$ is defined over $\widetilde{K}_p$ by our
discussion in (\ref{4.3}), it follows from [\cite{Bo}, AG, 14.5]
that so is $Y_p^c=\,\overline{\omega_p(Y_p)}$. But then the
composite $G_{p}/T_p^{\,{\rm
gen}}\stackrel{\sim}{\to}Y_p\stackrel{\omega_p}\too Y_p^c$ is a
smooth projective model of $G_{p}/T_p^{\,{\rm gen}}$ defined over
$\widetilde{K}_p=K_p(\mathcal{T}_p)$.
\end{proof}
\begin{rem}\label{div}
Since the variety $Y^c_p$ is projective and its open set
$\omega_p(Y_p)\cong Y_p$ is affine, all irreducible components of
the complement $D\,:=\,Y^c_p\setminus \omega_p(Y_p)$ have
codimension $1$ in $Y^c$; see [\cite{Ha}, Ch.~2].
\end{rem}
\begin{rem}
Let $\bf T$ be a maximal $F$-torus in a split connected reductive
algebraic $F$-group $\bf G$. If $\bf T$ is $F$-split, then there
exist Borel subgroups ${\bf B}_+$ and ${\bf B}_-$ in $\bf G$ defined
over $F$ and such that ${\bf B}_+\cap {\bf B}_-={\bf T}$. Thus, in
the $F$-split case the variety $({\bf G}/{\bf B}_{+}) \times ({\bf
G}/{\bf B}_{-})$ provides a natural smooth projective model of the
homogeneous space ${\bf G}/{\bf T}$ (this was pointed out to me by
Panyushev and Serganova). Unfortunately, it is not clear how to
adapt this construction to the case of a non-split maximal torus. It
would be very interesting to find an {\it explicit} $F({\bf
T})$-defined smooth projective model of the homogeneous space ${\bf
G}/{\bf T}$ for an arbitrary maximal $F$-torus $\bf T$ of $\bf G$.
\end{rem}
\subsection{}\label{4.6.}
It is known that in characteristic $0$ the generic torus $T^{\rm
gen}\subset G$ splits over a finite Galois extension of
$\widetilde{K}$ whose group acts on the weight lattice $X(T^{\rm
gen})$ as the Weyl group $W$. This result is sometimes attributed to
{\'E}. Cartan; see [\cite{C1}, \cite{C2}]. Modern proofs can be
found in [\cite{Vo1}, \cite{Vo2}] and in the ``dismissed appendix''
to [\cite{CT}] written by J.-L. Colliot-Th{\'e}l{\`e}ne; see
[\cite{CT1}].

The rest of the paper relies on a modular version of this result. In
order to apply the arguments from [\cite{Vo1}, \cite{Vo2}] in the
characteristic $p$ case one needs to know that the morphism
$\alpha\colon\,G/T\times T\to H$ is separable and the second
projection $\pi\colon\,H\to Y$ is birational (notation of {\it
loc.\,cit.}). This was checked earlier by Vladimir Popov and Andrei
Rapinchuk. The proof below was outlined to the author by Andrei
Rapinchuk.
\begin{prop}\label{galois}
There exists a finite Galois extension $L/\widetilde{K}_p$ with
group $W$ which splits the $\widetilde{K}_p$-torus $T^{\rm gen}_p$
and  acts on the weight lattice of $T^{\rm gen}_p$ in the standard
way.
\end{prop}
\begin{proof}
\noindent (1) Following [\cite{Vo1}] we set
$$H_p:=(gN_{p},gtg^{-1})\,|\,\,g\in G_{p},\,t\in T_{p}\}\subset\,(G_{p}/N_{p})\times G_{p}.$$
By [\cite{Hu1}, p.~10], the set $H_p$ is Zariski closed in
$(G_{p}/N_{p})\times G_{p}$. By the definition of  $H_p$, the
$K_p$-morphism $\alpha\colon\,(G_{p}/T_{p})\times T_{p}\too H_p$
taking $(gT_{p},t)$ to $(gN_{p},gtg^{-1})$ is surjective.  We can
write $\alpha$ as the composition $\alpha_2\circ\alpha_1$, where
\begin{eqnarray*}
\alpha_1\colon\,G_{p}\times T_{p} &\too& (G_{p}/T_{p}) \times G_{p},
\qquad\ \, (g,t) \mapsto (gT_{p},
gtg^{-1});\\
\alpha_2\colon\,(G_{p}/T_{p}) \times G_{p} &\too&
(G_{p}/N_{p})\times G_{p},\  \qquad (gT_{p},x) \mapsto (gN_{p},x).
\end{eqnarray*}
The morphism $\alpha_2$ is an {\'e}tale Galois cover. We need to
show that the second projection $\pi_2\colon\,H_p\to G_{p}$ is a
separable morphism. Define $\alpha_3\colon\,(G_{p}/T_{p})\times
T_{p}\too G_{p}$ to be the composite of $\alpha_1$ and $\pi_2$. Then
$\alpha_3(gT_{p},t)=gtg^{-1}$ for all
$(gT_{p},t)\in(G_{p}/T_{p})\times T_{p}$.

We compute the differential of $\alpha_3$ at $(eT_{p},t_0)\in
(G_{p}/T_{p})\times T_{p}$, where $e$ is the identity element of
$G_{p}$ and $t_0$ is any regular element of $T_{p}$. Write
$D=\K[\varepsilon]$ for the algebra of double numbers over $\K$, so
that $\varepsilon^2=0$. Let $X$ be in the $\K$-span of the
$e_{\gamma}$'s which we identify with $\g_{p}/\t_{p}$, the tangent
space of $G_{p}/T_{p}$ at $eT_{p}$. If $Y\in \Lie(T_{p})=\t_{p}$,
then $t_0Y$ lies in the tangent space to $T_{p}$ at $t_0$ and $({\rm
d}\alpha_3)_{(eT_{p},\,t_0)}(X,t_0Y)$ is the coefficient of
$\varepsilon$ of the element $(e + \varepsilon X)(t_0 + \varepsilon
t_0Y)(e - \varepsilon X) = t_0 + \varepsilon(t_0Y + Xt_0 - t_0X)\in
D.$ Multiplying this by $t_0^{-1}$ to move everything back to the
identity, we get
$$({\rm
d}\alpha_3)_{(eT_{p},\,t_0)}(X,t_0Y)\,=\,Y + t_0^{-1}Xt_0 - X = Y +
\big({\rm Ad}(t_0^{-1})-{\rm Id}\big)(X).$$ Since $t_0\in T_{p}$ is
regular, we see that the image of $({\rm d}\alpha_3)_{(X,\,t_0Y)}$
has dimension $n=\dim\, G_{p}$. So $\alpha_3=\alpha_1\circ\pi_2$ is
a separable morphism and hence so is $\pi_2$.

\medskip

\noindent (2)  Write $G^{\rm rs}_{p}$ for the set of all regular
semisimple elements in $G_{p}$, and set $T_{p}^{\rm rs}:=G_{p}\cap
T_{p}$. As the group $G_{p}$ is simply connected it follows from
Steinberg's restriction theorem that there exists a regular
invariant function $f\in\k[G_\k]^{G_\k}$ such that $G_{p}^{\rm
rs}\,=\,\{g\in G_{p}\,|\,\,f(g)\ne 0\}$. Hence $G_{p}^{\rm rs}$ is a
principal Zariski open subset in $G_{p}$. In particular, the
varieties $G_{p}^{\rm rs}$ and $T_{p}^{\rm rs}$ are smooth and
affine.

Let $H^{\rm rs}_p:=\{(gN_{p},gtg^{-1})\,|\,\,g\in G_{p},t\in
T_{p}^{\rm rs}\}$. The Weyl group $W$ acts on $(G_{p}/T_{p})\times
T_{p}$ by the rule
$$(gT_{p},t)^w=\,(g\!\stackrel{\cdot}{w}\!T_{p},
{\stackrel{\cdot}{w}}{}^{-1}t\!\stackrel{\cdot}{w}).$$ It is
straightforward to see that the set $H^{\rm rs}_p$ is $W$-stable,
the restriction of $\pi_2$ to $H^{\rm rs}_p$ is bijective and the
fibres of $\alpha$ are $W$-orbits. Also, $H^{\rm
rs}_p=\big((G_{p}/N_{p})\times G_{p}^{\rm rs}\big)\cap H_p$ is a
principal Zariski open subset of $H_p$.

By part~(1), the restriction of $\pi_2$ to $H^{\rm rs}_p$ is
separable. So $\pi_2\colon H^{\rm rs}_p\to G_{p}^{\rm rs}$ is a
bijective separable morphism of affine varieties. Therefore, it is
birational. As the variety $G_{p}^{\rm rs}$ is smooth, hence normal,
Zariski's Main Theorem now yields that $\pi_2\colon H^{\rm rs}_p\to
G_{p}^{\rm rs}$ is a $K_p$-isomorphism. But the $H^{\rm rs}_p$ is an
affine normal variety, and we can apply [\cite{Bo}, Prop.~II.\,6.6]
to conclude that $\alpha\colon\,(G_p/T_p)\times T^{\rm rs}_{p}\too
H_p^{\rm rs}$ is the geometric quotient for the action of $W$.

We have the following commutative diagram, where $\pi_1$ is the
first projection and $\beta$ is the canonical map.
\begin{eqnarray}\label{cd}
\begin{CD}
(G_{p}/T_{p})\times T^{\mathrm{rs}}_{p} @>\alpha>> H^{\rm rs}_p\\
@Vp_1VV @VV\pi_1V \\
 G_{p}/T_{p}@>\beta>> G_{p}/N_{p}@<\sim<< \mathcal{T}_p,
\end{CD}
\end{eqnarray}
where $p_1$ and $\pi_1$ are the first projections and $\beta$ is the
quotient morphism. All these maps are defined over $\k\subset K_p$.

\medskip

\noindent (3) Recall from (\ref{4.2}) that $\widetilde{y}_p$
identifies with a generic point of $\mathcal{T}_p\cong G_{p}/N_{p}$
in the sense that the fields $K_p(\widetilde{y}_p)=\widetilde{K}_p$
and $K_p(G_{p}/N_{p})$ are $K_p$-isomorphic. The algebra of
$K_p$-defined regular functions of the fibre
$\beta^{-1}(\widetilde{y}_p)$ is
\begin{equation}\label{gg}
K_p[G_{p}/T_{p}]
\otimes_{\,K_p[G_{p}/N_{p}]}\,K_p(G_{p}/N_{p})\,\cong\,
K_p(G_{p}/T_{p}), \end{equation} showing that
$K_p(\beta^{-1}(\widetilde{y}_p))=K_p(G_{p}/T_{p})$ is a Galois
extension of $\widetilde{K}_p=K_p(G_{p}/N_{p})$ with Galois group
$W$.

Since $T^{\rm gen}_p\,=\,Z_{G_{p}}(\widetilde{y}_p)$ is defined over
$\widetilde{K}_p$, it contains a $\widetilde{K}_p$-rational regular
element; we call it $s$. (This follows from the fact that $T^{\rm
gen}_p(\widetilde{K}_p )$ is dense in $T^{\rm gen}_p$; see
[\cite{DG}]). Let $(gN_{p},s)=\pi_2^{-1}(s)$. Since $\pi_2$ is a
$K_p$-isomorphism, we have that $(gN_{p},s)\in H_p^{\rm
rs}(\widetilde{K}_p)$.

If $P\subset \K$ is a finite Galois extension of $\widetilde{K}_p$
which splits $T_p^{\rm gen}$, then the split $P$-tori $T_{p}$ and
$T^{\rm gen}_p$ are conjugate by a $P$-rational element of $G_{p}$;
see [\cite{BoT}, 4.21, 8.2]. Put differently,
$(gN_{p},s)=\alpha(hT_{p},s')$ for some $P$-point $(hT_{p},s')$ of
$(G_{p}/T_{p})\times T_{p}^{\rm rs}$. But then (\ref{cd}) shows that
$\widetilde{y}_p\in \beta((G_{p}/T_{p})(P))$. Conversely, if $L$ is
a finite Galois extension of $\widetilde{K}_p$ such that
$\widetilde{y}_p\in \beta((G_{p}/T_{p})(L))$, then (\ref{cd}) yields
that
$$(gN_{p},s)=\pi_2^{-1}(s)\in
\alpha((G_{p}/T_{p})\times T_{p})(L)).$$ Hence $L$ splits $T_p^{\rm
gen}$. Applying this with $L=K_p(G_{p}/T_{p})$ and taking into
account (\ref{gg}) one can observe that $L=K_p(G_{p}/T_{p})$ is a
minimal splitting field for $T_p^{\rm gen}$ and ${\rm
Gal}(L/\widetilde{K}_p)=W$. Indeed, since $L^W=\widetilde{K}_p$ and
$W$ acts faithfully on $L$, the $L$-algebra $L\,\otimes_{L^W}L$ is
isomorphic to a direct sum of $|W|$ copies of $L$. Moreover, it
follows from the normal basis theorem by comparing $W$-invariants
that if $F/\widetilde{K}_p$ is a Galois extension contained in $L$,
then $L\otimes_{L^W} F$ is isomorphic as an $F$-algebra to a direct
sum $|W|$ copies of $F$ if and only if $F=L$.

By the minimality of the splitting field $L$, the Galois group of
$L/\widetilde{K}_p$ acts faithfully on the weight lattice $X(T^{\rm
gen}_p)$ giving a natural injective group homomorphism
$\tau\colon\,{\rm Gal}(L/\widetilde{K}_p)\to {\rm Aut}(\Phi)$. Since
the group $G_{p}$ is $\widetilde{K}_p$-split, the image of ${\rm
Gal}(L/\widetilde{K}_p)$ under $\tau$ is contained in $W\subseteq
{\rm Aut}(\Phi)$; see [\cite{Ti}, 2.3]. As $\tau$ is injective, this
shows that $W={\rm Gal}(L/\widetilde{K}_p)$ acts on $X(T^{\rm
gen}_p)$ in the standard way.
\end{proof}
\subsection{}
It what follows we shall assume without loss of generality that our
algebraic closure $\K$ of $\widetilde{K}_p=K_p(G_{p}/N_{p})$
contains $L=K_p(G_{p}/T_{p})$. The result below (which is crucial
for us) has been proved in [\cite{CT}] under the assumption that the
base field has characteristic $0$; compare [\cite{CT}, Thm~6.3(b)].
Although it follows from a more general result obtained in
[\cite{CTK}], the proof given in [\cite{CT}] is self-contained
modulo [\cite{Po}, Thm.~4], [\cite{CS}, Prop.~2.1.1] and [\cite{CS},
Prop.~2.A.1].

Recall that a free $\Z$-module of finite rank acted upon by a group
$\Gamma$ is called a {\it permutation lattice} if it has a
$\Z$-basis whose elements are permuted by $\Gamma$.
\begin{prop}\label{lattice}
If the homogeneous space $Y_p=G_{p}/T^{\rm gen}_p$ is
$\widetilde{K}_p$-rational, then there exists a short exact sequence
of $\Gamma$-lattices $$0\too P_2\too P_1\too X(T^{\rm gen}_p)\too
0$$ with $P_1$ and $P_2$ permutation lattices over $\Gamma={\rm
Gal}(L/\widetilde{K}_p)$.
\end{prop}
\begin{proof}
The proof repeats almost verbatim the argument in [\cite{CT},
p.~23]; we sketch it for the convenience of the reader.

By Proposition~\ref{proj}, the variety $Y_p$ has a smooth projective
model $Y_p^c$ defined over $\widetilde{K}_p$. The open immersion
$Y_p\subset Y_p^c$ gives rise to an exact sequence of Galois
lattices
\begin{equation}\label{short}
0\too \K[Y_p]^{\times}/\K^{\times}\too{\rm Div}_{\infty}\, Y_p^c\too
{\rm Pic}\,Y_p^c\too {\rm Pic}\,Y_p\too 0,\end{equation} where ${\rm
Div}_\infty\,Y_p^c$ is the free abelian group on the irreducible
components of the exceptional divisor $D=Y_p^c\setminus Y_p$; see
Remark~\ref{div}. Since [\cite{Po}, Thm.~4] and [\cite{CS},
Prop.~2.1.1] hold in any characteristic, we can repeat the argument
in [\cite{CT}, p.~23] to obtain that $\K[Y_p]^\times=\,\K^\times$
and ${\rm Pic}\,Y_p\,\cong\,X(T_p^{\rm gen})$ as Galois modules.

Since [\cite{CS}, Prop.~2.A.1] hold in any characteristic, we can
apply it to the $\widetilde{K}_p$-rational homogeneous space $Y_p$
to deduce that the Galois lattice ${\rm Pic}\,Y_p^c$ has the
property that $Q_1\oplus {\rm Pic}\,Y_p^c\,\cong\, Q_2$ for some
permutation Galois lattices $Q_1$ and $Q_2$. As $P:={\rm
Div}_\infty\,Y_p^c$ is a permutation Galois lattice as well, the
short exact sequence
$$0\too P\too {\rm Pic}\,Y_p^c\too X(T_{p}^{\rm gen})\too 0$$
induced by (\ref{short}) gives rise to a short exact sequence $0\to
P_2\to P_1\to X(T_{p}^{\rm gen})\to 0$ with $P_2=P\oplus Q_1$ and
$P_1=Q_1\oplus {\rm Pic}\,Y_p^c$ being permutation Galois lattices.
\end{proof}
We denote by $P(\Phi)$ the weight lattice of the root system $\Phi$.
\begin{corollary}\label{W}
If the Gelfand--Kirillov conjecture holds for $\g$, then there
exists a short exact sequence of $W$-modules
\begin{equation}\label{last}
0\too P_2\too P_1\too P(\Phi)\too 0 \end{equation} with $P_1$ and
$P_2$ permutation $W$-lattices.
\end{corollary}
\begin{proof} Combining Theorem~\ref{pur} and Proposition~\ref{gen}
we see that if the Gelfand--Kirillov conjecture holds for $\g$, then
for all $p\gg 0$ the homogeneous space $Y_p=G_{p}/T_p^{\rm gen}$ is
rational over $\widetilde{K}_p$. As $G_{p}$ is simply connected,
$X(T_p^{\rm gen})\cong P(\Phi)$ as $W$-modules. Now the result
follows by applying Propositions~\ref{galois} and~\ref{lattice}
\end{proof}
\subsection{} We are finally ready for the main result of this
paper.
\begin{theorem} Let $\g$ be a finite dimensional simple Lie algebra
over an algebraically closed field of characteristic $0$ and assume
that $\g$ is not of type ${\rm A}_n$, ${\rm C}_n$ of ${\rm G}_2$.
Then the Gelfand--Kirillov conjecture does not hold for $\g$.
\end{theorem}
\begin{proof}
By [\cite{CT}, Prop.~7.1], it follows from the existence of a short
exact sequence $0\to P_2\to P_1\to P(\Phi)\to 0$ with $P_1$ and
$P_2$ permutation $W$-lattices that $\Phi$ is of type ${\rm A}_n$,
${\rm C}_n$ of ${\rm G}_2$. Applying Corollary~\ref{W} finishes the
proof.
\end{proof}

\end{document}